\documentclass[oneside,reqno,11pt]{amsart}
\usepackage{epsfig}
\usepackage{color}
\usepackage{soul}
\usepackage{amssymb,amsmath,amsthm,amstext,amsfonts}
\usepackage[normalem]{ulem}
\usepackage{amsmath,amscd,amstext,amsthm,amsfonts,latexsym}
\usepackage[colorlinks=true,linkcolor=blue, urlcolor=red, citecolor=blue,backref]{hyperref}	
\usepackage{amsmath,amssymb,amsthm}
\usepackage{color}
\usepackage[active]{srcltx}
\usepackage{amssymb, amsmath}
\usepackage{graphicx}
\usepackage{float}
\usepackage[active]{srcltx}

\usepackage[all,arc]{xy}
\usepackage{enumerate}
\usepackage{mathrsfs}

\newtheorem{theorem}{Theorem}[section]
\newtheorem{corollary}[theorem]{Corollary}
\newtheorem{proposition}[theorem]{Proposition}
\newtheorem{lemma}[theorem]{Lemma}

\theoremstyle{definition}
\newtheorem{maintheorem}{Theorem}

\newtheorem{maincorollary}[maintheorem]{Corollary}

\newtheorem{definition}[theorem]{Definition}

\newtheorem{example}[theorem]{Example}

\newtheorem{claim}{Claim}

\newtheorem{remark}[theorem]{Remark}

\newcommand{\RR}{\mathbb R}

\newcommand{\htop}{h_{\topp}}

\newcommand{\vep}{\varepsilon}

\newcommand{\Leb} {\text{Leb}}

\newcommand{\N}{\mathbb{N}}
\newcommand{\R}{\mathbb{R}}

\newcommand{\diam}{\operatorname{diam}}

\newcommand{\topp}{\operatorname{top}}

\newcommand{\hgeo}{h_{\text{geom}}}

\newcommand{\cF}{\mathcal{F}}

\newcommand{\cU}{\mathcal{U}}

\newcommand{\dM} {\text{dist}_M} 
\renewcommand{\le} {\leqslant} 
\renewcommand{\leq} {\leqslant} 
\renewcommand{\ge} {\geqslant} 
\renewcommand{\geq} {\geqslant}

\usepackage[normalem]{ulem}

\newcounter{main}

\makeatletter
\let\c@equation\c@theorem
\makeatother
\numberwithin{equation}{section}

\title{Variational principle for the Carathéodory--Pesin entropy of foliations}

\author[A. Bi\'s]{Andrzej Bi\'s}
\address{Andrzej Bi\'s, Faculty of Mathematics, {\L}\'od\'z University, Poland}
\email{andrzej.bis@wmii.uni.lodz.pl}

\author[M. Carvalho]{Maria Carvalho}
\address{Maria Carvalho, CMUP \& Departamento de Matem\'atica, Faculdade de Ci\^encias da Universidade do Porto, Rua do Campo Alegre 687, Porto, Portugal.}
\email{mpcarval@fc.up.pt}

\author[P. Varandas]{Paulo Varandas}
\address{Paulo Varandas, Center for Research and Development in Mathematics and Applications (CIDMA), Department of Mathematics, University of Aveiro, 3810--193 Aveiro, Portugal.}
\email{paulo.varandas@ua.pt}


\date{\today}

\begin{document}

\begin{abstract}

We introduce Carathéodory–Pesin structures for non-singular foliations on Riemannian manifolds and define a new notion of entropy for such foliations. Under dynamical and geometric assumptions on the corresponding holonomy pseudogroups, we establish a variational description of this new concept which relates it to the local metric entropy of measures on transversals and to the entropy of harmonic measures on the manifold.





\end{abstract}

\maketitle


\section{Introduction}\label{se:introd}

The study of entropy for foliated manifolds lies at the intersection of several fields, such as dynamical systems, differential geometry, and geometric topology. A foliation decomposes the manifold into immersed submanifolds, called leaves, and its complexity is imparted by the way nearby leaves separate when followed along leafwise paths. This viewpoint inspired Ghys, Langevin and Walczak to define the notion of geometric entropy of a foliation (see \cite{GLW}).
This motivates the question, raised in \cite{GLW}, of whether the geometric entropy can be related by a variational principle to a suitable measure-theoretic entropy of probability measures on the ambient manifold. A natural class of probability measures on foliated manifolds is given by harmonic measures, introduced by Garnett \cite{Garnett}, which are stationary under leafwise heat diffusion and always exist on compact foliated manifolds. In foliated boxes, they admit a distinguished disintegration whose conditional measures along plaques (these are the connected local slices of the leaves of the foliation) are absolutely continuous with respect to the leafwise volume, and the densities are harmonic maps (cf. \cite{Candel, Garnett}). Thus, harmonic measures are well suited to the leafwise geometry of the foliation. 

\smallskip

The previous notion of entropy for a foliation is geometric in nature, although it can be expressed in terms of the topological entropy of holonomy pseudogroups acting on sections transversal to the foliation (cf. \cite{GLW, Walczak}). Passing from the ambient foliated manifold to transverse sections leads to the study of entropy for saturated sets, since every holonomy invariant set on transversals gives rise to a foliated invariant set. The purpose of this paper is to address this issue through Carath\'eodory--Pesin structures and transverse dynamics. Following the general dimension-theoretic formalism of Pesin \cite{Pe97}, we introduce a Carath\'eodory--Pesin entropy for saturated subsets of a foliated manifold (cf. Section~\ref{se:setting}). The construction of this concept is based on leafwise dynamical balls defined using the geometric distances described by Ghys, Langevin and Walczak. This covering formalism is flexible enough for us to handle saturated sets and, no less important, it can be compared with the corresponding Carath\'eodory--Pesin entropy of holonomy pseudogroups on transversals. 

\smallskip 
Our first result states that if $\mathcal F$ is a non-singular foliation of a closed compact Riemannian manifold $M$ admitting a global transversal, the Carathéodory--Pesin entropy of the entire foliation, denoted by $h_{CP}(M,\mathcal F)$, satisfies 
$$
h_{CP}(M,\mathcal F) \,=\,
\sup_{\mathcal U}\,
\frac{
h_{X_{\mathcal U}}(G_{\mathcal U},G_{\mathcal U,1})
}
{\operatorname{diam}(\mathcal U)}
$$
where the supremum is taken over all finite nice foliated atlases $\mathcal U$, $X_{\mathcal U}$ is an associated global transversal, and $(G_{\mathcal U},G_{\mathcal U,1})$ is the holonomy pseudogroup determined by $\mathcal U$. Combining this information with the variational principle for finitely generated pseudogroups established in \cite{BCV}, we obtain the transversal variational formula
$$
h_{CP}(M,\mathcal F) \,=\,
\sup_{\mathcal U}\,
\sup_{\nu\,\in\,\mathcal P(X_{\mathcal U})}\,
\frac{
\underline h_\nu(G_{\mathcal U},G_{\mathcal U,1})
}
{\operatorname{diam}(\mathcal U)}
$$
where $\mathcal P(X_{\mathcal U})$ stands for the space of Borel probability measures on the global transversal $X_{\mathcal U}$ and $\underline h_\nu(G_{\mathcal U},G_{\mathcal U,1})$ denotes the local metric entropy of $\nu$ with respect to the holonomy pseudogroup $G_{\mathcal U}$ and a fixed set of generators $G_{\mathcal U, 1}$. 
\smallskip

In order to relate this transversal formalism to harmonic measures on the manifold, and inspired by the local entropy introduced by Dinh, Nguyen and Sibony in the setting of hyperbolic Riemann surface laminations (cf. \cite{DNS}), we define the local entropy of a harmonic measure as the exponential decay of the measure of leafwise dynamical balls with the radius. Moreover, given a finite nice foliated atlas $\mathcal U$ and a harmonic measure $\mu$, this pointwise decay rate 
may be integrated on a complete transversal with respect to the transverse measure induced by $\mu$. For tame harmonic measures (that is, harmonic measures whose leafwise densities are uniformly bounded from above and from below in finite coverings by foliated boxes, we show that the measure of a leafwise dynamical ball can be compared with the measure of a holonomy dynamical ball on the transversal; and so, if $\mu$ is tame and $\nu_{\mathcal U}$ is its induced probability measure on $X_{\mathcal U}$, then
$$
h_\mu(M,\mathcal F) \,=\,
\sup_{\mathcal U}\,
\frac{
\underline h_{\nu_{\mathcal U}}
(G_{\mathcal U},G_{\mathcal U,1})
}
{\operatorname{diam}(\mathcal U)}.
$$
This equality identifies the entropy of tame harmonic measures with the local metric entropy of their induced transverse measures. Under an additional liftability assumption  
(that is, assuming that every probability measure on a global transversal can be induced by some tame harmonic measure on the foliated manifold), we establish the following variational principle
$$
h_{CP}(M,\mathcal F)
\,=\,
\sup
\left\{
h_\mu(M,\mathcal F) \colon 
\mu\in\mathcal P_H^{\mathrm{tame}}(M,\mathcal F)
\right\}
$$
where $P_H^{\mathrm{tame}}(M,\mathcal F)$ denotes the space of tame harmonic probability measures on $(M, \mathcal F)$. This liftability hypothesis is essential for the proof but is not valid in general: in Section~\ref{se:thmE}, we provide an example of a smooth foliation with a global transversal and a probability measure on this transversal that is not induced by any harmonic probability measure on the ambient manifold.

\smallskip
The paper is organized as follows. In Section~\ref{se:setting} we recall the definition of geometric entropy of a foliation, introduce the Carath\'eodory--Pesin entropy of saturated sets, and define the local entropy with respect to a foliation of any probability measure on the manifold.  Section~\ref{se:statements} presents the statements of our main results. In Section~\ref{sec:prelim} we gather essential information on finitely generated pseudogroups, holonomy pseudogroups and their Carath\'eodory--Pesin entropy. In Section~\ref{sec:proof-entropy-holonomy} we prove the variational principle relating the Carath\'eodory--Pesin entropy of the foliation to the corresponding transverse entropies. Section~\ref{sec:harmonicm} reviews harmonic measures and their disintegrations, just before we prove the non-liftability result. In Section~\ref{sec:measure-geometric-entropy-F} we show that the local entropy of a measure is independent of the chosen complete transversal. In Section~\ref{sec:proof-harmonic-measure-entropy} we establish the variational principle for the entropy of tame harmonic measures over the local metric entropy of the induced transverse measures. Finally, in Section~\ref{sec:examples} we discuss a few examples, including suspended foliations, foliations generated by flows, and holonomy invariant measures with zero entropy.
\section{Setting}\label{se:setting}

In this section, we introduce the notions of entropy of a foliation and entropy of a harmonic measure with respect to a foliation. In addition, we set the necessary notation.

\subsection{Geometric entropy of a foliation}
Let $M$ be a compact, closed Riemannian manifold and $\dM$ be the Riemannian distance on $M$. Given a $C^1$-foliation $\mathcal F$ and $r>0$, denote by $\mathcal F_x$ the leaf of the foliation $\mathcal F$ containing $x\in M$, by $B_{\mathcal F}(x,r)$ the ball centered at $x$ with radius $\vep$ contained in the leaf $\mathcal F_x$ with respect to the induced distance, and by $\Leb_x$ the volume measure on the leaf $\mathcal F_x$. 

\smallskip

By compactness of $M$, there exists a finite open nice covering $\mathcal U =\{U_1, U_2, \dots, U_k\}$ of $M$ by foliated charts (see Definition~\ref{def:foliated-atlas}), and there is $\vep_0>0$ such that, for every $i \in \{1, \dots, k\}$ and $x, y\in U_i$ such that $dM(x,y)<\vep_0$, the point $x$ can be projected orthogonally (in a unique way) to a point in the plaque $P_y \subset U_i$ containing the point $y$. In particular, for every $x,y\in U_i$ such that $\dM(x,y)<\vep_0$, and every leaf curve $\gamma \colon [0,1] \to \mathcal F_x \cap U_i$ such that $\gamma(0)=x$, there exists a uniquely defined orthogonal projection $\Pi_y\gamma \colon \,[0,1] \to \mathcal F_y$. In case $\dM(\gamma(1), \Pi_y\gamma(1))<\vep_0$, and changing the chart if necessary, one can continue to extend such paths until a time $t_0$ such that the distance between endpoints of the extended curves exceeds $\vep_0$ (see \cite[Subsection 3.3]{Walczak} for more information). 

\smallskip

In what follows, the diameter of a covering $\cU$ is denoted by $\diam (\cU)=\sup_{j} \diam U_j$, where 
$$\diam U_j=\sup_{P\,\in\, U_j}\, \sup_{x,\,y\,\in\, P}\, d_P(x,y)$$
and $d_P$ stands for the induced Riemannian metric on the plaque $P$. The length of a curve $\gamma$ will be abbreviated to $\ell(\gamma)$.

\smallskip

\begin{definition}\cite{GLW}\label{eq:Dmetric}
\emph{For each $r>0$, consider the metric
$$D_r(x,y)\,=\,\max \Big\{\delta_r(x,y), \,\, \delta_r(y,x),\,\,\dM(x,y)\,\Big\}$$
defined for all points $x,y\in M$ such that $dM(x,y)< \vep_0$, where
$$\delta_r(x,y)\, =\, \sup_\gamma \max_{t\,\in\, [0,\,r]} \,\,\dM(\gamma(t), \Pi_y(\gamma(t))$$
and the supremum is taken over all curves $\gamma$ inside $\mathcal F_x$ such that $\gamma(0)=x$ and $\ell(\gamma) \leq r$.}
\end{definition}
 
\smallskip

The family of distances $\delta_r$ yields a natural way to define separation of points in $M$.

\begin{definition}
\emph{Given $r>0$ and $0<\vep<\vep_0$, two points $x,y\in M$ are \emph{$(\mathcal F,r,\vep)$-separated} if one of the following conditions holds:
\begin{itemize}
\item[$(a)$] $\dM(x,y)>\vep_0$.
\item[$(b)$] There exists a curve $\gamma \colon [0,1] \to \mathcal F_x$ such that $\gamma(0)=x$, $\ell(\gamma)\le r$ and 
$$\dM(\gamma(1),\,\Pi_y(\gamma(1)))\,>\,\vep.$$
\item[$(c)$] There exists a curve $\gamma \colon [0,1] \to \mathcal F_y$ such that $\gamma(0)=y$, $\ell(\gamma)\le r$ and 
$$\dM(\gamma(1),\,\Pi_x(\gamma(1)))\,>\,\vep.$$
\end{itemize}}

\emph{A subset $E\subset X$ is \emph{$(\mathcal F,r,\vep)$-separated} if any two points in $E$ are $(\mathcal F,r,\vep)$-separated. We denote by 
$s(\mathcal F, r,\vep)$ the largest cardinality of $(\mathcal F,r,\vep)$-separated sets in $M$.}
\end{definition}

The complexity of a foliation of $M$ will be quantified precisely by the exponential rate of separation of points of $M$ along the leaves of the foliation.

\begin{definition}\cite{GLW}\label{def:geomPressure}
\emph{The \emph{geometric entropy of the foliation $\mathcal F$} in $M$ is the limit 
\begin{equation}\label{eq:geomH}
\hgeo(M,\mathcal F) \, =\, \lim_{\vep\,\to\,0^+}\,\limsup_{r\,\to\, +\infty}\, \frac1r \log s(\mathcal F,r,\vep).
\end{equation}}
\end{definition}

The geometric entropy of a foliation is a purely geometric concept, though it can be expressed in terms of the topological entropy of the holonomy pseudogroups on transversal directions (cf. \cite[Th\'eor\`eme 3.4]{GLW}). 
\medskip


\subsection{Carath\'eodory-Pesin entropy of a foliation}
Let us introduce another notion of entropy of a foliation using Carath\'eodory-Pesin structures for $\mathcal F$-saturated sets. A subset $Z\subset M$ is said to be $\mathcal F$-saturated if it is a union of leaves. Given such a set $Z$, $r>0$ and $\varepsilon>0$, denote by
$$B_r(x,\varepsilon) \,=\,\{y\in M \colon \, D_r(x,y)\,<\,\varepsilon\}$$
the \emph{dynamical ball} with respect to the distance $D_r$ (see Definition~\ref{eq:Dmetric}). 
\smallskip

Fix $r_0>0$ and $\varepsilon>0$. We denote by
$C_Z(r_0,\varepsilon)$ the family of all finite or countable sets
$I_{r_0}(\varepsilon)\subset Z\times [r_0,+\infty)$ 
such that
$$Z \,\subset\,
\bigcup_{(x_j,\,r_j)\,\in\, I_{r_0}(\varepsilon)}\,\,
B_{r_j}(x_j,\varepsilon).$$ 
For notational simplicity, we shall write $j\in I_{r_0}(\varepsilon)$
to identify the pair $(x_j,r_j)\in I_{r_0}(\varepsilon)$.
For $s\geq0$, define
\begin{equation*}\label{eq:MZh}
M_Z(s,\varepsilon,r_0)
=
\inf
\left\{
\sum_{j\,\in\, I_{r_0}(\varepsilon)}
e^{-s r_j}
\colon \,
I_{r_0}(\varepsilon)\in C_Z(r_0,\varepsilon)
\right\}.
\end{equation*}
The limit
$M_Z(s,\varepsilon)
=
\lim_{r_0\to+\infty}
M_Z(s,\varepsilon,r_0)$ 
exists because the function $r_o \in \mathbb R_+ \mapsto  M_Z(s,\varepsilon,r_0)$ is increasing. Furthermore,  
$s \geq 0 \mapsto M_Z(s,\varepsilon)$ 
is decreasing and, by the Carath\'eodory-Pesin construction, there exists a critical value $h_Z(\mathcal F,\varepsilon) \ge 0
$ 
such that
$$M_Z(s,\varepsilon) \,=\,
\begin{cases}
    \begin{array}{cl}
        +\infty & \quad \text{if }s<h_Z(\mathcal F,\varepsilon) \\
         0 & \quad \text{if }s>h_Z(\mathcal F,\varepsilon)
    \end{array}
\end{cases}.
$$
In other words,
\begin{equation*}\label{eq:formdim}
h_Z(\mathcal F,\varepsilon)
=
\sup\Big\{s\geq0 \colon \,\ M_Z(s,\varepsilon)=+\infty\Big\}
=
\inf\Big\{s\geq0 \colon \,\ M_Z(s,\varepsilon)=0\Big\}.
\end{equation*}
We note that the map $\vep > 0 \mapsto h_Z(\mathcal F,\varepsilon)$ is monotonic.
\smallskip

\begin{definition}
\emph{The \emph{geometric Carath\'eodory-Pesin entropy of the foliation $(M,\mathcal F)$ restricted to $Z$} is 
defined by
$$h_{\mathrm{CP}}(Z,\mathcal F) \,=\,
\lim_{\varepsilon\,\to\,0^+}\, h_Z(\mathcal F,\varepsilon).$$} 
\end{definition}

\subsection{Tame harmonic measures and tame foliated manifolds}

Let $\mathcal P(M)$ denote the space of Borel probability measures on $M$. 

\begin{definition}\cite{Garnett}\label{def:harmonic} \emph{A Borel probability measure $m\in \mathcal P(M)$ is called \emph{harmonic} if 
$$\int_M \,\Delta_{\mathcal F} f(x) \; dm(x)\,=\,0$$
for every continuous map $f\colon M \to \mathbb{R}$ that is $C^2$ on the leaves, 
where $\Delta_{\mathcal F}$ denotes the leaf Laplacian operator. In what follows, $\mathcal P_{H}(M,\mathcal F)$ stands for the space of harmonic measures of the foliated manifold $(M,\mathcal F)$.}
\end{definition}

It is known that any foliation of a compact manifold admits a finite nice foliated atlas, hence a nice covering, and that any nice covering determines a finitely generated holonomy pseudogroup acting on a  complete  transversal (see Subsection~\ref{sec:foliationsandhpseudo} for more details). The role in this setting of harmonic measures is due to the fact that they admit absolutely continuous disintegrations along plaques of the foliation, in the following sense: if $m$ is a harmonic then, for every nice foliated atlas $\cU$ with a complete   transversal $X_\cU$, one has
\begin{equation}\label{eq:harmU1}
m \mid_\cU  \,\,=\,\int_{X_\cU} \,m_x \,\, d\nu(x)
\end{equation}
where $\nu$ is a Borel probability measure on $X_\cU$ and $m_x$ is absolutely continuous with respect to the Lebesgue measure on the plaque that contains $x$. We refer the reader to Proposition~\ref{prop:disintegrationharm} for a precise statement.

\begin{definition}\label{def:tame}
\emph{A harmonic measure $m\in \mathcal P(M)$ is said to be \emph{tame} if, for $\nu$-almost every $x \in X_\mathcal U$, the (Radon-Nikodym) density of $\mu_x$ with respect to the Lebesgue measure on the plaque that contains $x$ is uniformly bounded from above and from below. We denote by $\mathcal P^{\text{tame}}_H(M,\mathcal F)$ the class of tame harmonic measures on $M$.} 
\end{definition}

\smallskip

Tame harmonic measures suggest the following class of foliated manifolds.

\begin{definition}\label{def:tamefoliation}
\emph{Given a Riemannian manifold $M$ and a $C^1$ non-singular foliation $\mathcal F$ of $M$, we say that the pair $(M,\mathcal F)$ is a \emph{tame foliated space} if, for every nice foliated atlas $\cU$ of $M$, every  complete  transversal $X_\cU$ and every Borel probability measure $\nu \in \mathcal P(X_\cU)$, there exists a tame harmonic measure $\mu\in \mathcal P_H(M,\mathcal F)$ such that \eqref{eq:harmU1} holds.}
\end{definition}

\smallskip
The class of tame foliated spaces includes suspended foliations as a special example (cf. Lemma~\ref{le:lift-suspension}). As far as we know, it is not clear whether all foliated manifolds are tame.

\smallskip

\subsection{Geometric entropy of a harmonic measure}

Let us now recall the notion of geometric entropy of $\mu \in \mathcal P(M)$ as introduced by Dinh, Nguyen and Sibony in the context of hyperbolic foliations (cf. \cite{DNS}). Consider the family of maps $\delta_r$ and metrics $D_r$ introduced in 
Definition~\ref{eq:Dmetric} and, given $r,\vep>0$ and $x\in M$, set 
\begin{equation*}\label{def:dynball-delta}
B_{r}(x,\vep) \,=\,\{y\in M \colon D_r(x,y) < \vep\}.
\end{equation*}
The \emph{local entropy of $\mu\in \mathcal P(M)$ at $x\in M$} is defined by
\begin{equation}\label{def:localFmeasureentropy}
h_\mu(M,\mathcal F,x)\,=\, \lim_{\vep\,\to\,0^+}\,\liminf_{r\,\to\,+\infty} \,-\frac1r \log \mu(B_r(x,\vep)).
\end{equation}
(Actually, Dinh, Nguyen and Sibony use $\limsup$ instead of $\liminf$ in   ~\eqref{def:localFmeasureentropy}.)
\smallskip

 

Given a nice foliated atlas $\mathcal U=(U_i)_{1\,\le\, i\, \le\, k}$ with local transversals $(X_{U_i})_{1\,\le\, i\, \le\, k}$ and a complete transversal 
$$X_\cU =\bigsqcup_{1\,\le\, i\, \le\, k}\, X_{U_i}$$
\emph{the local entropy of $\mu$ with respect to the atlas $\mathcal U$} is defined by
\begin{equation}\label{eq:defhmuU}
h_\mu(M,\cF,\cU)\,=\,\int_{X_\cU} \,h_\mu(M,\mathcal F,x)\,\, d\nu_{X_\cU}(x)
\end{equation} 
where $\nu_{X_\cU}$ is the probability induced by $\mu$ on $X_\cU$ (cf. \eqref{eq:harmU1}). We note that $h_\mu(M,\cF,\cU)$ does not depend on the transversal $X_\cU$ (cf. Lemma~\ref{le:changetransv}).
\smallskip

\begin{definition}\cite{DNS}\label{def:geomentr}
\emph{The \emph{geometric entropy of $\mu\in \mathcal P(M)$ with respect to the foliation $\mathcal F$ of $M$} is defined by 
\begin{equation}\label{def:geometricmeasureentropy}
h_\mu(M,\cF) \,=\, \sup_{\mathcal U} \; \big\{h_\mu(M,\cF,\mathcal U) \big\}
\end{equation}
where the supremum is taken over all nice coverings $\mathcal U$ of $M$.}
\end{definition}


\section{Main results}\label{se:statements}

This work provides a variational description of the Carathéodory--Pesin entropy of foliations of arbitrary dimension on compact, closed Riemannian manifolds. More precisely, we relate the Carathéodory--Pesin entropy of the foliation to the Carathéodory--Pesin entropy of the holonomy pseudogroups induced by finite nice foliated atlases. Combining this comparison with the variational principle for finitely generated pseudogroups obtained in \cite{BCV}, we establish a variational principle expressed in terms of probability measures on  complete  transversals. 
In general, it is not a variational principle over harmonic measures. 
\smallskip

Given a finite nice foliated atlas $\mathcal U$, denote by $X_{\mathcal U}$ the  corresponding complete  transversal and by $(G_{\mathcal U},G_{\mathcal U,1})$ the holonomy pseudogroup determined by $\mathcal U$ with the generator set $G_{\mathcal U,1}$. In what follows, $h_{X_{\mathcal U}}(G_{\mathcal U},G_{\mathcal U,1})$ stands for the Carath\'eodory-Pesin entropy of this pseudogroup on $X_{\mathcal U}$ and $\underline{h}_{\nu}(G_{\mathcal U},G_{\mathcal U,1})$ is the  local metric entropy of $\nu \in \mathcal{P}(X_{\mathcal U})$ 
(cf. Subsections~\ref{subsec:entropyPG} and \ref{sec:foliationsandhpseudo} for precise definitions).


\begin{maintheorem}
\label{thm:entropy-holonomy}
\emph{Let $\mathcal F$ be a $C^1$ non-singular foliation of a compact, closed Riemannian manifold $M$ with a global transversal and $Z\subset M$ be an $\mathcal F$-saturated set. For every finite nice foliated atlas $\mathcal U$, define $Z_{\mathcal U}=Z\cap X_{\mathcal U}$,  
where $X_{\mathcal U}$ the corresponding complete transversal. Then
$$h_{\mathrm{CP}}(Z,\mathcal F)
=
\sup_{\mathcal U}
\left\{
\frac{
h_{Z_{\mathcal U}}(G_{\mathcal U},G_{\mathcal U,1})
}
{\operatorname{diam}(\mathcal U)}
\right\}.
$$
If, in addition, $Z$ is Borel, then
$$h_{\mathrm{CP}}(Z,\mathcal F)
\,=\,
\sup_{\mathcal U}
\sup_{\substack{
\nu\,\in\,\mathcal P(X_{\mathcal U})\\
\nu(Z_{\mathcal U})\,=\,1
}}
\left\{
\frac{
\underline h_\nu(G_{\mathcal U},G_{\mathcal U,1})
}
{\operatorname{diam}(\mathcal U)}
\right\}.
$$
In particular,
\begin{equation}\label{eq:CP}
h_{\mathrm{CP}}(M,\mathcal F) \,\, = \,\,\sup_{\mathcal U} \,\left\{\frac{h_{X_{\mathcal U}}(G_{\mathcal U},G_{\mathcal U,1})}{\operatorname{diam}(\mathcal U)}\right\} \,\,=\,\,\sup_{\mathcal U} \,\sup_{\nu\, \in\, \mathcal{P}(X_{\mathcal U})} \,\left\{
\frac{\underline{h}_{\nu}(G_{\mathcal U},G_{\mathcal U,1})}{\operatorname{diam}(\mathcal U)}\right\} 
\end{equation}
where the first supremum is taken over all finite nice foliated atlases $\mathcal U$ of $M$ and the second is taken over all Borel probability measures on $X_{\mathcal U}$.}
\end{maintheorem}

The first equality in \eqref{eq:CP} provides a version of \cite[Th\'eor\`eme 3.4]{GLW} within the framework of Carath\'eodory-Pesin entropy, while the second is a variational principle that connects the topological complexity of the foliation (measured by $h_{\mathrm{CP}}(M,\mathcal F)$) with the  local metric entropy of Borel probability measures on global transversals. The second equality is a direct consequence of the first together with the variational principle for pseudogroups established in \cite{BCV} (see Theorem~\ref{thm:variationalpseudogroups}).
Observe also that, by \cite[Th\'eor\`eme~3.4]{GLW} and \cite[Theorem~B]{BCV}, one has \begin{eqnarray*}
h_{geom}(M, \mathcal F) &=& \sup_{\mathcal U} \,\left\{\frac{h_{top}(G_{\mathcal U},G_{\mathcal U,1})}{\operatorname{diam}(\mathcal U)}\right\} \\
&&\smallskip \\
h_{top}(G_{\mathcal U},G_{\mathcal U,1}) &\geq& \sup_{\nu\, \in\, \mathcal{P}(X_{\mathcal U})} \,\big\{\,
\underline{h}_{\nu}(G_{\mathcal U},G_{\mathcal U,1})\,\big\} \quad \forall\, \mathcal U,
\end{eqnarray*}
and so
$$h_{geom}(M, \mathcal F) \,\geq\, h_{\mathrm{CP}}(M,\mathcal F).$$
\smallskip

In \cite{GLW}, Ghys, Langevin and Walczak asked whether the geometric entropy of a foliation can be related to harmonic probability measures on the ambient manifold by using a suitable notion of measure-theoretic entropy. The
variational principle in Theorem~\ref{thm:entropy-holonomy} does not answer this question since it concerns the dynamics on transversals and arbitrary transversal probability measures cannot, in general, be lifted to harmonic probability measures on the foliated manifold, even for smooth minimal foliations
(cf. Theorem~\ref{thm-non-liftable}).
\smallskip

A comment is due regarding the assumption that the foliation has a global transversal. In the proof of Theorem~\ref{thm:entropy-holonomy} we use \cite[Lemma~3.7]{GLW}, which concerns $(n,\vep)$-separated sets or, equivalently, \emph{balls} with respect to the map
$$\tau_n(x,y) \, =\, \max\big\{d(g(x), g(y)\colon \, g \in G_n\big\}$$
which estimates the separation of points by the holonomy pseudogroup (see the precise definitions in Section~\ref{sec:prelim}) but may not be a metric. However, we also use \cite[Theorem~B]{BCV}, which requires a true dynamical metric $d_n$, like the one defined in \cite[Section~2]{BCV}. The existence of a global transversal implies that the holonomy pseudogroup is a group, so $\tau_n = d_n$ for every $n \in \N$, thereby preventing these incompatibilities.
\smallskip

The next result is a measure-theoretic version of Theorem~\ref{thm:entropy-holonomy} for harmonic measures.

\begin{maintheorem}
\label{thm:harmonic-measure-entropy}
\emph{Let $\mathcal F$ be a $C^1$ non-singular foliation of a compact Riemannian manifold $M$ and $\mu \in \mathcal P^{\text{tame}}_H(M,\mathcal F)$ be a tame harmonic measure. Given a finite nice foliated atlas $\mathcal U$, denote by $\nu_{\mathcal U}$ the probability measure induced by $\mu$ on a  global transversal   $X_{\mathcal U}$. Then
$$h_\mu(M, \mathcal F) \,=\,\sup_{\mathcal U}\,\left\{
\frac{\underline{h}_{\nu_{\mathcal U}}(G_{\mathcal U},G_{\mathcal U,1})}{\operatorname{diam}(\mathcal U)}\right\} \,=\, \lim_{\diam \mathcal U \,\to \,0^+}\,\left\{\frac{\underline{h}_{\nu_{\mathcal U}}(G_{\mathcal U},G_{\mathcal U,1})}{\operatorname{diam}(\mathcal U)}\right\}.$$
In particular,  $h_\mu(M, \mathcal F) \, \leq \, h_{\mathrm{CP}}(M, \mathcal F)$  for every $\mu \in \mathcal P^{\text{tame}}_H(M,\mathcal F)$.}
\end{maintheorem}

Combining the previous two results with the variational principle for finitely generated pseudogroups established in \cite[Theorem~B]{BCV}, we obtain a variational principle for tame foliated spaces. 

\begin{maincorollary}\label{cor:tamefoliated}
\emph{Let $\mathcal F$ be a $C^2$ non-singular foliation of a closed compact Riemannian manifold $M$ with a global transversal. Assume that $(M,\mathcal F)$ is a tame foliated space. 
Then
$$h_{CP}(M,\mathcal F) \,=\, \sup \left\{
h_\mu(M, \mathcal F) \colon\,\, 
\mu\in\mathcal P_H^{\mathrm{tame}}(M,\mathcal F)
\right\}.$$}
\end{maincorollary}

The previous results ensure that the Carath\'eodory-Pesin entropy of a foliation is zero if and only if the geometric entropy of every tame harmonic measure vanishes. Moreover, it shows that harmonic measures provide an adequate class of measures to detect the complexity of a foliation. In contrast, holonomy invariant measures may be too restrictive: even when they exist, they may fail to carry the full entropy of the holonomy pseudogroup, and hence of the foliation  (see Example~\ref{ex:holonomy-invariant-zero-entropy}).

\smallskip

\begin{maincorollary}\label{cor:zeroentropy}
\emph{Let $\mathcal F$ be a $C^2$ non-singular foliation of a closed, compact, Riemannian manifold $M$ with a global transversal. Assume that $(M,\mathcal F)$ satisfies the tame liftability property. Then
$$h_{\mathrm{CP}}(M,\mathcal F) 
\,=\, \sup \left\{
h_\mu(M,\mathcal F) \colon\,\, 
\mu\in\mathcal P_H^{\mathrm{tame}}(M,\mathcal F)
\right\}.$$}
\end{maincorollary}



The next result justifies the assumption previously made that the foliated manifold satisfies the tame liftability property. Moreover, it shows that the transversal variational principle that we established cannot, in general, be converted into a variational principle over harmonic measures. 

\begin{maintheorem}
\label{thm-non-liftable}
\emph{There exists a smooth foliation $(M,\mathcal F)$ with a global
transversal $\Sigma$ and a probability measure $\nu\in\mathcal P(\Sigma)$
such that no harmonic probability measure on $(M,\mathcal F)$ induces
$\nu$ on $\Sigma$.}
\end{maintheorem}



\section{Preliminaries}\label{sec:prelim}

In this section, we gather more detailed information regarding foliations, complete and global transversals, and holonomy pseudogroups.

\subsection{Finitely generated pseudogroup actions}\label{sec:prelim1}

The concept of pseudogroup is essential in the study of the geometry of foliated manifolds, and was introduced by Haefliger (cf. \cite{Hae2}). As shown in \cite{GLW}, the complexity of a foliation can be studied using a special class of finitely generated pseudogroups, called holonomy pseudogroups. In this subsection, we recall what is a finitely generated pseudogroup of local homeomorphisms acting on a compact metric space  and, having fixed a set of generators, how Ghys, Langevin and Walczak define in \cite{GLW} its topological entropy.
\smallskip

Given a compact metric space  $(X,d)$, let $\text{Homeo}(X)$ denote the family of homeomorphisms between open subsets of $X$. 
For $g \in Homeo(X)$ denote by $D_g$ its domain and by $R_g=g(D_g)$ its range.

\begin{definition} \emph{A collection $G\subset \text{Homeo}(X)$ is a \emph{pseudogroup} if it satisfies the following conditions:}

\noindent [Composition] \emph{If $g,f \in G$, then $g \circ f \colon  f^{-1}(R_f \cap D_g)\to g(R_f \cap D_g)$ belongs to $G$.}
\\

\noindent [Symmetry] \emph{If $g \in G,$ then $g^{-1}\in G$.}
\\

\noindent [Identity] \emph{The identity map $id_X$ is in $G$ (so $\bigcup_{g\,\in\, G}\,(D_g\cup R_g)=X$).}
\\

\noindent [Restriction to open sets] \emph{If $g \in G$ and $W \subset D_g$ is an open subset of $X$, then $g|_W \in G$.}
\\

\noindent [Gluing condition] \emph{If $g\colon D_g \to R_g$ is a homeomorphism between open subsets of $X$ and, for each $p \in D_g$, there exists a neighborhood $V$ of $p$ in $D_g$ such that $g|_V \in G,$ then $g \in G.$}
\end{definition}

For any set $G_1 \subset \text{Homeo}(X)$ which satisfies the condition 
$$\bigcup_{g \,\in\, G_1} \{D_g \cup R_g\}\,=\,X$$
there exists a unique (smallest in the sense of inclusion) pseudogroup $G$ that contains $G_1.$ This is referred to as the \emph{pseudogroup generated by $G_1$}. If the set $G_1$ is finite, we say that $G$ is \emph{finitely generated}. We note that, for such a pseudogroup $G$, one has $g \in G$ if and only if $g \in \text{Homeo}(X)$ and for any $x \in D_g$ there exist maps $g_1,...,g_k \in G_1,$ exponents $e_1,...,e_k \in \{-1,1\}$ and an open neighborhood $V$ of $x$ in $X$ such that
$$V \subset D_g \quad \text{and} \quad g|_V \,=\,g_1^{e_1}\circ \ldots \circ g_k^{e_k}|_V.$$
In particular, holonomy pseudogroups associated to a finite atlas of foliated charts on a Riemannian manifold $M$ are finitely generated (cf. \cite{Thu2}).

\smallskip

In what follows, we fix a pseudogroup $G$ acting on a compact metric space $(X,d)$ with a finite symmetric generating set
$$G_1=\{\text{id}\mid_X,g_1,g_1^{-1},g_2,g_2^{-1},\ldots,g_k,g_k^{-1}\}.$$
For any positive integer $n$, define
$$G_n=\{g_{i_1} \circ \ldots \circ g_{i_n} \colon \,  g_{i_j} \in G_1\}$$
whenever the composition $g_{i_1} \circ \ldots \circ g_{i_n}$ has a non-empty domain.
For each $x\in X$, set
$$G_n^x =\{g\in G_n \colon x\in D_{g}\}.$$
Since, by assumption, $id_X \in G_1$, we get $G_m \subset G_n$ whenever $m \le n.$ To emphasize the dependence of the previous concept, and further results, on the generating set, sometimes we write $(G,G_1)$ instead of simply $G.$

\smallskip

Given $x\in X$, $n \in \N$ and $\vep>0$, define
\begin{equation}\label{eq:defbynballG}
B^G_{n}(x, \vep)\,=\,\bigcap_{g\,\in\, G_n} \,g^{-1}\big(B(g(x),\vep\big) \cap D_g
\end{equation}
where $D_g\subset X$ is the domain of $g$.

\subsection{Topological entropy of a pseudogroup}\label{subsec:entropyPG}

There are in the literature several proposals to generalize the notion of topological entropy for a single dynamics to the setting of finitely generated semigroup actions; for an account on some of them, we refer the reader to \cite{CRV,XM} and references therein. In this subsection, we recall the notion of topological entropy of a pseudogroup introduced by Ghys, Langevin and Walczak in \cite{GLW}, which we denote by $\htop(G,\,G_1)$ to emphasize its dependence on the set $G_1$ of generators.

\smallskip

\begin{definition}\cite{GLW}
\emph{Given $n \in \N$ and $\vep>0$, we say that two points $x,y \in X$ are $(n,\vep)-$separated by $(G,G_1)$ whenever there exists $g\in G_n$ such that $x,y \in D_g$ and $d(g(x),g(y)) \ge \vep.$ Let $s(n,\vep)$ denote 
the maximal number of $(n,\vep)$-separated points of $X$. The \emph{topological entropy of $(G, G_1)$} is defined by
\begin{equation*}\label{def:entropy-fg}
\htop(G,G_1)\,=\,\lim_{\vep\, \to\, 0^+}\,\limsup_{n\, \to\,+\infty}\,
\frac{1}{n}\log s(n,\vep)
\end{equation*}}
\end{definition}

It is known that, if $G_1$ and $G'_1$ are generating sets of the pseudogroup $G$ and the topological entropy $\htop(G,G_1)=0,$ then the topological entropy $\htop(G,G'_1)$ vanishes as well (cf. \cite{GLW}).

\subsection{Carath\'eodory-Pesin entropy of a pseudogroup}\label{subsec:CPentropyPG}

Consider a finitely generated pseudogroup $(G,G_1)$ generated by a finite set $G_1$ of local homeomorphisms of a compact metric space $(X,d)$. Following \cite{BCV}, fix an arbitrary subset $Z \subset X.$ Given $N\in \mathbb N$ and $\vep>0$ denote by $I_N(\vep)\subset Z \times \{n\in \mathbb N\colon n\geqslant N\}$ a finite or countable set such that $Z$ is covered by dynamical balls $B^G_{n_j}(x_j, \vep)$, where $(x_j,n_j) \in I_N(\vep)$ and $n_j \geqslant N$. In other words,
$$
Z \,\subset\, \bigcup_{(x_j,n_j) \,\in\, I_N(\vep)}B^G_{n_j}(x_j, \vep).
$$
Denote by $C_Z(N, \vep)$ the family of all such subsets $I_N(\vep)$. For notational simplicity, we will write $j\in I_N(\vep)$ to identify the pair $(x_j,n_j) \in I_N(\vep)$.

\smallskip
For $s\geq0$, $N\in\mathbb N$ and $\vep>0$, define
$$
M_Z^G(s,\vep,N)
\,=\,
\inf
\Big\{
\sum_{j\,\in\, I_N(\vep)}
e^{-s n_j}
\colon \,
I_N(\vep)\in C_Z(N,\vep)
\Big\}.
$$
Since the family of admissible covers (that is, covers by dynamical balls $B^G_{n_j}(x_j, \vep)$, where $(x_j,n_j) \in I_N(\vep)$ and $n_j \geqslant N$) decreases as $N$ increases, the function
$N\mapsto M_Z^G(s,\vep,N)$ 
is non-decreasing, and so the limit
$$
M_Z^G(s,\vep)
\,=\,
\lim_{N\,\to\,+\infty}\,
M_Z^G(s,\vep,N)
$$
exists. It is not hard to check, using the classical approach in \cite{Pe97}, that there is a well defined value
$$h_Z(G,G_1,\vep)\,=\,
\sup\Big\{s\geq0 \colon \, M_Z^G(s,\vep) = +\infty\Big\}\,=\,
\inf\Big\{s\geq0 \colon \,M_Z^G(s,\vep)=0\Big\}$$
and that the map $\vep>0 \mapsto h_Z(G,G_1,\vep)$ is monotonic.

\begin{definition}
\emph{The \emph{Carath\'eodory-Pesin entropy of $(G,G_1)$ restricted to $Z$} is defined by
$$h_Z(G,G_1) \,=\,\lim_{\vep\,\to\,0^+}\,
h_Z(G,G_1,\vep).$$}
\end{definition}

\smallskip

There is an alternative way, described in  \cite{Pe97}, of defining the Carath\'eodory-Pesin entropy of a pseudogroup by using finite or countable open covers. Fix an arbitrary subset $Z$ of a compact metric space $(X,d)$. Given $N\in \mathbb N$ and a finite open cover $\mathcal O=(O_i)_{i}$ of $X$, define
\begin{equation*}\label{eqdefCil}
\mathcal O^{(N)}\,= \bigvee_{g_0,\, g_1,\, \ldots,\, g_{N-1}\, \in\, G_1}\,\, \bigvee_{j=0}^{N-1}\, (g_{j}\circ \ldots \circ g_1\circ g_0)^{-1}(\mathcal O).
\end{equation*}
An element in $\mathcal O^{(N)}$ is said to be length $N$. Moreover, two points $x,y$ belong to an element $O\in \mathcal O^{(N)}$ if and only if $g_{j}\circ \dots \circ g_1\circ g_0(x)$ and $g_{j}\circ \dots \circ g_1\circ g_0(y)$ belong to the same element of $\mathcal O$ for every $g_i\in G_1$ and any $0\le j\le N-1$. Let $C_Z(N, \mathcal O)$ be the family of all subcovers of $\bigcup_{n\,\ge\, N} \mathcal O^{(n)}$, whose parameterization is denoted by $I_N(\mathcal O)$, in such a way that 
$$Z \,\subset\, \bigcup_{j \,\in\, I_N(\mathcal O)} O_{n_j}(x_j)$$
where $n_j \geqslant N$ and $x_j\in O_{n_j}(x_j)$.
\smallskip

Fix an open cover $\mathcal O$ of $X$. For $s\geq0$ and $N\in\mathbb N$, define
$$
M_Z^G(s,\mathcal O,N)
\,=\,
\inf
\Big\{
\sum_{j\,\in\, I_N(\mathcal O)}
e^{-s n_j}
\colon \,
I_N(\mathcal O)\in C_Z(N,\mathcal O)
\Big\}.
$$
Since the family of admissible covers decreases as $N$ increases, the function
$N\mapsto M_Z^G(s,\mathcal O,N)$ 
is non-decreasing, and so the limit
$$M_Z^G(s,\mathcal O) \,=\,\lim_{N\,\to\,\infty}\,
M_Z^G(s,\mathcal O,N)$$ 
exists. It is not hard to check (details in \cite{Pe97}) that there is a well defined value
$$h_Z(G,G_1,\mathcal O)
\,=\,
\sup\Big\{s\geq0 \colon \, M_Z^G(s,\mathcal O)\,=\,+\infty\Big\}
\,=\,
\inf\Big\{s\geq0 \colon \,M_Z^G(s,\mathcal O)\,=\,0\Big\}.
$$
Moreover (cf. \cite[Section 11]{Pe97}) the Carath\'eodory-Pesin entropy of $(G,G_1)$ restricted to $Z$ can be expressed as 
\begin{equation*}\label{eq:Pesincoverings}
h_Z(G,G_1)
\,=\,
\lim_{\diam \mathcal O \,\to\,0^+}\,
h_Z(G,G_1,\mathcal O)    
\end{equation*}
and
\begin{equation*}\label{eq:Pesincoverings2}
\htop(G,G_1)
\,=\,
\lim_{\diam \mathcal O \,\to\,0^+}\,
\limsup_{n\,\to\, +\infty}\,
\frac1n \,
\log N_n(\mathcal O)
\end{equation*}
where, for every finite open cover $\mathcal U$ of $X$, the number $N_n(\mathcal O)$ is the minimal cardinality of a cover of $X$ by elements of $\mathcal O^{(n)}$.

\subsection{A partial variational principle for pseudogroups}

We start by defining the \emph{local metric entropy} of a Borel probability measure $\nu \in \mathcal P(X)$.

\begin{definition}
\emph{Given $\nu \in \mathcal P(X)$, the \emph{ local metric entropy of $\nu$ with respect to the pseudogroup $(G, G_1)$} is the average
\begin{equation*}\label{eqQ:deflocalent}
\underline h_\nu((G,G_1))
\,=\, \int_X\, \lim_{\vep\,\to\, 0^+}\,\liminf_{n\,\to\, +\infty}\,
-\frac1n \,\log \nu (B^G_n(x,\vep)) \; d\nu(x)
\end{equation*}
where $B_n^G(x,\vep)$ is defined in ~\eqref{eq:defbynballG}.}
  
\end{definition}

\begin{theorem}\cite[Theorem~B]{BCV}\label{thm:variationalpseudogroups} Let $(G,G_1)$ be a finitely generated pseudogroup acting on a compact metric space $(X,d)$. Then:
\begin{itemize}
\item[(a)] For every Borel subset $Z$ of $X$,
$$h_Z((G,G_1)) \,=\, \sup\Big\{\underline{h}_{\nu}(G,G_1) \colon \,\nu \in \mathcal P(X) \text{ and } \nu(Z)=1\Big\}.$$
\item[(b)] $\htop(G,G_1) \,\geqslant \, h_X(G,G_1).$ 
\end{itemize}
\end{theorem}

\smallskip

\subsection{ Foliations and holonomy pseudogroups} \label{sec:foliationsandhpseudo}

The formal definition of a $C^r$-foliation, where $1\le r\le \infty,$ is as follows.

\begin{definition} \emph{A $p$-dimensional $C^r$-foliation $\mathcal F$ on a $n$-dimensional manifold $M$ is a decomposition of $M$ into connected submanifolds $\{\mathcal F_{\alpha}\}_{\alpha \, \in\, A},$ called \emph{leaves of the foliation $\mathcal F$}, such that for any point $x \in M$ there exist a neighborhood $V$ of $x$ in $M$ and a $C^r$-differentiable chart 
$\phi=(\phi_1,\phi_2) \colon V \to \RR^n =\RR^p \times \RR^{n-p}$ with the following property: for any leaf $\mathcal F_{\alpha}$ of $\mathcal F$, the connected components of $\mathcal F_{\alpha} \cap V$
are described by the equation $\phi_2(\cdot)=v$, for $v\in \mathbb R^{n-p}$. The connected components of $\mathcal F_{\alpha} \cap V$ are called \emph{plaques}. Such foliation $\mathcal F$ has codimension $n-p$. }
\end{definition}

To connect foliations and pseudogroups, we need to address nice coverings.

\begin{definition}\label{def:foliated-atlas}
\emph{We say that a foliated manifold $(M,\mathcal F)$  admits a \emph{nice foliated atlas} $\mathcal{A}$ (or that the cover of $M$ by the domains $V_\phi$ of the charts $\phi \in \mathcal{A}$ is nice) if the following conditions hold:
\begin{enumerate}
\item[$(a)$] The cover $\{V_\phi g \colon \, \phi \in \mathcal{A}\}$ is locally finite.
\item[$(b)$] For any chart $\phi \in \mathcal{A}$, the range $\phi(V_\phi) \subset \RR^n$ is an open cube.
\item[$(c)$] For any $\phi, \psi \in \mathcal{A}$ satisfying the condition $V_\phi \cap V_\psi \neq  \emptyset$ there exists a chart $\varphi$ distinguished by the foliation $\mathcal F$ such that $\varphi(V_\varphi)$ is an open cube, $V_\varphi$ contains the closure of $V_\phi \cup V_\psi$ and $\phi = \varphi|_{V_\phi}.$
\end{enumerate}}
\end{definition}

It is known that any $C^1$ non-singular foliation of a compact Riemannian manifold admits a finite nice covering and that any nice covering $\mathcal{U}$  determines a finitely generated holonomy pseudogroup. Moreover, by compactness of $M$, any nice atlas of a $p$-dimensional foliation $\mathcal F$ admits a finite nice subatlas $\mathcal{A}$. We refer the reader to \cite[Chapter 1]{Walczak} for proofs of the previous assertions.

\smallskip

Since the range $\phi(V_\phi)$ of any $\phi \in \mathcal{A}$ in an open cube in $\RR^n,$ it can be written as a product $U_1(\phi)\times U_2(\phi)$, 
where $U_1(\phi)$ and $U_2(\phi)$ are open
cubes in $\RR^p$ and $\RR^{n-p}$, respectively. 
Therefore, each $U_1(\phi)\times\{y_0\},$ where
$y_0 \in U_2(\phi),$ is isomorphic to a plaque in $V_\phi$, and the points of $U_2(\phi)$ parameterize the plaques in $V_\phi.$

\begin{definition}\label{def:transversals}
\emph{Given a chart $\phi$ and a point $x_0 \in U_1(\phi)$, the submanifold
$$T_\phi \,=\,\phi^{-1}\big(\{x_0\} \times U_2(\phi)\big)$$ 
of $V_\phi$ intersects every plaque of $V_\phi$ exactly once. It is called a \emph{local transversal} of $V_\phi.$}
\end{definition}

\smallskip
The domains of the finite nice atlas $\mathcal{A}$ constitute a finite nice covering
$\mathcal{U}=\{U_1,...,U_k\}$ of $M$ which determines a finitely generated holonomy
pseudogroup $(G_{\mathcal{U}}, G_{\mathcal{U},1})$ we now describe. Take $U_i,U_j \in \mathcal{U}$ with local transversals $T_i \subset U_i$ and
$T_j \subset U_j.$ Assume that $U_i \cap U_j\neq \emptyset$. Then, for any $x \in U_i \cap U_j$ there exist a unique plaque $P_i(x)\in U_i$ and a unique plaque $P_j(x)\in U_j$ such that $x \in P_i(x) \cap P_j(x).$ 

\begin{definition}\label{def:holonomy}
\emph{Given a finite nice atlas $\mathcal A$, 
and $U_i,U_j\in\mathcal U$ with $U_i\cap U_j\neq\emptyset$, the local
holonomy transformation
$ 
g_{U_i,U_j}
$ 
is the map which sends the point $T_i\cap P_i(x)$ to the point
$T_j\cap P_j(x)$, whenever these two plaques meet in $U_i\cap U_j$.
The local holonomy map $g_{U_i,U_j}$ is a homeomorphism between open subsets of the local transversals $T_i$ and $T_j$. The finite set 
$$G_{\mathcal{U},1}\,=\,\big\{g_{_{U_i,U_j}} \colon \, U_i,U_j \in \mathcal{U}~and ~U_i \cap U_j \neq \emptyset \big\}$$
generates the \emph{holonomy pseudogroup $G_{\mathcal{U}}$ determined by the nice covering $\mathcal{U}$.}}
\end{definition}

Regarding the domains of the maps $g_{_{U_i,U_j}}$ of the pseudogroup, we use the following notation. After shrinking the foliated boxes if necessary, we choose the local transversals with compact closure inside the corresponding boxes. We shall denote by
$$
X_{\mathcal U}
=
\bigsqcup_{i=1}^k \overline{T_i}
$$
the associated compact transversal, endowed with the induced metric.

\begin{definition}\label{def:global}
\emph{The pseudogroup $G_{\mathcal{U}}$ acts on the {disjoint} union 
$$
X_{\mathcal U}
=
\bigsqcup_{i=1}^k \overline{T_i}
$$
which is called a \emph{complete transversal} of $\mathcal F$. We shall say that the foliation $\mathcal F$ admits a global holonomy model if the complete transverse section $X_{\mathcal U}$ can be chosen so that the generators of the holonomy pseudogroup extend to homeomorphisms of the entire space $X_{\mathcal U}$.}
\end{definition}
 
We note that any two holonomy pseudogroups of a $C^1$ non-singular foliation $\mathcal F$ on a closed Riemannian manifold $M$ 
are isomorphic (cf. \cite{Hae2,Hae3}).


\section{Proof of Theorem~\ref{thm:entropy-holonomy}}\label{sec:proof-entropy-holonomy}

We start by proving the equality
\begin{equation}\label{eq:equalitythmA}
h_{\mathrm{CP}}(Z,\mathcal F)
\,=\,
\sup_{\mathcal U}
\left\{
\frac{
h_{Z_{\mathcal U}}(G_{\mathcal U},G_{\mathcal U,1})
}
{\operatorname{diam}(\mathcal U)}
\right\}.    
\end{equation}
Fix a finite nice foliated atlas $\mathcal U$ and write
$ 
d_{\mathcal U}:=\operatorname{diam}(\mathcal U).
$ 
Since $Z$ is $\mathcal F$-saturated, the set
$ 
Z_{\mathcal U}=Z\cap X_{\mathcal U}
$ 
is invariant under the holonomy pseudogroup
$(G_{\mathcal U},G_{\mathcal U,1})$. Indeed, if
$x\in Z_{\mathcal U}$ and $g(x)$ is defined for some
$g\in G_{\mathcal U}$, then $g(x)$ belongs to the same leaf as $x$.
Since $Z$ is saturated, $g(x)\in Z$, and hence
$g(x)\in Z_{\mathcal U}$.
We prove the two inequalities separately.
\medskip

\medskip
\noindent $(\geq)\quad $ We now show that
$$
h_{\mathrm{CP}}(Z,\mathcal F)
\geq
\frac{
h_{Z_{\mathcal U}}(G_{\mathcal U},G_{\mathcal U,1})
}
{d_{\mathcal U}}.
$$
More precisely, for every $s>0$ such that 
$$
0<s<
\frac{
h_{Z_{\mathcal U}}(G_{\mathcal U},G_{\mathcal U,1})
}
{d_{\mathcal U}}.
$$
we will prove that $s\leq h_{\mathrm{CP}}(Z,\mathcal F)$.
Fix $r_0>0$ and $0<\varepsilon<\varepsilon_0/2$. Consider an arbitrary admissible cover of $Z$
$$Z\subset
\bigcup_{j\,\in\, I}
B_{r_j}(x_j,\varepsilon)$$
where $x_j\in Z$ and $r_j\geq r_0.$
For each $j\in I$, set
$$
n_j=\left\lfloor \frac{r_j}{d_{\mathcal U}}\right\rfloor.
$$
Then
$ 
n_j \,d_{\mathcal U}\,\leq\, r_j < (n_j+1)\,d_{\mathcal U}
$ 
and therefore
$$
B_{r_j}(x_j,\varepsilon)
\subset
B_{n_j d_{\mathcal U}}(x_j,\varepsilon).
$$
Let
$ J \,=\,\left\{
j\in I \colon B_{n_j \,d_{\mathcal U}}(x_j,\varepsilon)\cap Z_{\mathcal U}\neq\emptyset
\right\}
$ 
and, 
for each $j\in J$, choose
$ 
z_j\in
B_{n_j \,d_{\mathcal U}}(x_j,\varepsilon)\cap Z_{\mathcal U}.
$ 
Then
$$
Z_{\mathcal U} \subset \bigcup_{j\,\in\, J}
B_{n_jd_{\mathcal U}}(z_j,2\varepsilon).
$$
By the connection between foliated dynamical balls and holonomy dynamical balls established in \cite[Lemme~3.7]{GLW}, there exists
$C>0$ such that, for $\varepsilon>0$ small enough,
$$
Z_{\mathcal U}
\subset
\bigcup_{j\,\in\, J}
B_{n_j}^{G_{\mathcal U}}(z_j,\,2C\varepsilon).
$$
Moreover, if
$ 
N_0=\left\lfloor \frac{r_0}{d_{\mathcal U}}\right\rfloor
$ 
then $n_j\geq N_0$ for all $j\in J$. Thus, the previous family is an admissible cover of $Z_{\mathcal U}$ for the Carathéodory--Pesin
structure of the pseudogroup. Consequently,
$$
M^{G_{\mathcal U}}_{Z_{\mathcal U}}
(s \,d_{\mathcal U},\,2C\varepsilon,N_0)
\,\leq\,
\sum_{j\,\in\, J}
\,e^{-s \, d_{\mathcal U} \,n_j}.
$$
Since
$ 
r_j < (n_j+1)\,d_{\mathcal U},
$ 
we have
$ 
e^{-s\,r_j} \,\geq\,
e^{-s \,d_{\mathcal U}}\,e^{-s \,d_{\mathcal U}\,n_j}
$ 
and, consequently,
$$
\sum_{j\,\in\, I}\, e^{-s\,r_j}
\,\geq\,
e^{-s\, d_{\mathcal U}}
\,\sum_{j\,\in\, J}\,e^{-s\, d_{\mathcal U}\, n_j}
\,\geq\,
e^{-s \,d_{\mathcal U}}\,
M^{G_{\mathcal U}}_{Z_{\mathcal U}}
(s \,d_{\mathcal U},\,2C\varepsilon,N_0).
$$
Taking the infimum over all admissible covers of $Z$, we obtain
$$
M_Z(s,\varepsilon,r_0)
\geq
e^{-s \,d_{\mathcal U}}
\, M^{G_{\mathcal U}}_{Z_{\mathcal U}}
(s \,d_{\mathcal U},\,2C\varepsilon,N_0).
$$
Moreover, by the choice of $s$, one has
 $
s \,d_{\mathcal U} \,<\,
h_{Z_{\mathcal U}}(G_{\mathcal U},G_{\mathcal U,1}).
$ 
Therefore, for all sufficiently small $\varepsilon>0$ one has
$$
\lim_{N\,\to\, +\infty}
M^{G_{\mathcal U}}_{Z_{\mathcal U}}
(s \,d_{\mathcal U},\,2C\varepsilon,N)
\,=\,+\infty.$$
Since $N_0\to+\infty$ as $r_0\to+\infty$, it follows that
$$
M_Z(s,\varepsilon)
\,=\,
\lim_{r_0\,\to\, +\infty}\,M_Z(s,\varepsilon,r_0)
\,=\, +\infty.$$
Hence,
$ 
s\,\leq\, h_{\mathrm{CP}}(Z,\mathcal F,\varepsilon),
$ 
and letting $\varepsilon\to0^+$ we get
$ 
s\,\leq\, h_{\mathrm{CP}}(Z,\mathcal F).$ 
Since $s$ is arbitrary, this proves that
$$
h_{\mathrm{CP}}(Z,\mathcal F)
\,\geq\,
\frac{
h_{Z_{\mathcal U}}(G_{\mathcal U},G_{\mathcal U,1})
}
{d_{\mathcal U}}.
$$
Taking now the supremum over $\mathcal U$, we obtain
$$
h_{\mathrm{CP}}(Z,\mathcal F)
\,\geq\,
\sup_{\mathcal U}
\left\{
\frac{
h_{Z_{\mathcal U}}(G_{\mathcal U},G_{\mathcal U,1})
}
{\operatorname{diam}(\mathcal U)}
\right\}.
$$

\medskip

\noindent $(\leq)\quad $ Given
$$
s>
\sup_{\mathcal U}
\left\{
\frac{
h_{Z_{\mathcal U}}(G_{\mathcal U},G_{\mathcal U,1})
}
{\operatorname{diam}(\mathcal U)}
\right\}.
$$
we will prove that
$ 
h_{\mathrm{CP}}(Z,\mathcal F)\leq s.
$ 
Fix a finite nice foliated atlas $\mathcal U$. By the choice of $s$,
$ 
h_{Z_{\mathcal U}}(G_{\mathcal U},G_{\mathcal U,1})
<
s\,d_{\mathcal U}.
$ 
Take any real number $a$ such that
$$
h_{Z_{\mathcal U}}(G_{\mathcal U},G_{\mathcal U,1})
\,<\, a\,<\, s\,d_{\mathcal U}.
$$
Then, by the definition of the Carathéodory--Pesin critical value, for every $\eta>0$, for all sufficiently small $\varepsilon>0$ and all sufficiently large $N\in\mathbb N$, there exists a finite or countable cover
\begin{equation}\label{eq:cover}
Z_{\mathcal U}
\subset
\bigcup_{j\,\in\, J}\,
B_{n_j}^{G_{\mathcal U}}(z_j,\varepsilon),
\qquad
z_j\in Z_{\mathcal U},\quad n_j\geq N
\end{equation}
such that
$$
\sum_{j\,\in\, J}\,e^{-a n_j}\,<\,\eta.
$$
Since $\mathcal U$ is finite and $X_{\mathcal U}$ is a complete transversal, there exists $C_0\geq1$, depending only on $\mathcal U$, such that every point of $M$ can be joined inside its plaque to a point of $X_{\mathcal U}$ by a leaf curve of length at most 
$C_0\,d_{\mathcal U}$. 
\smallskip

\begin{claim}\label{cl:1}
\emph{$Z \subset \bigcup_{j\,\in\, J}\,
B_{(n_j+C_0)\,d_{\mathcal U}}(z_j,C\varepsilon)$, 
where $C>0$ is the constant given by \cite[Lemme~3.7]{GLW}.}
\end{claim}

\begin{proof}[Proof of Claim~\ref{cl:1}]
 Let $x\in Z$. Since $X_{\mathcal U}$ is a complete transversal, there exists a point $y\in X_{\mathcal U}$ lying in the same local plaque as $x$, with leafwise distance bounded by $C_0d_{\mathcal U}$. Since $Z$ is saturated and $x\in Z$, we have $y\in Z$. Hence
$ 
y\in Z_{\mathcal U}.
$ 
By \eqref{eq:cover}, there exists $j\in J$ such that
$$
y\in B_{n_j}^{G_{\mathcal U}}(z_j,\varepsilon).
$$
Using again the comparison between holonomy dynamical balls and foliated dynamical balls, and adding the plaque segment from $y$ to $x$, we get
$ 
x\in B_{(n_j+C_0)\,d_{\mathcal U}}(z_j,C\varepsilon).
$
\end{proof}

\smallskip

For each $j\in J$, define
$ 
r_j=(n_j+C_0)d_{\mathcal U}.
$ 
Taking $N$ large enough, we may assume that $r_j\geq r_0$ for every $j\in J$. So, the previous family is an admissible cover of $Z$ in the sense of the definition of $M_Z(s,C\varepsilon,r_0)$. Therefore,
$$M_Z(s,C\varepsilon,r_0) \,\leq\,
\sum_{j\,\in\, J}\,e^{-s\,r_j} \,=\,\sum_{j\,\in\, J}\,
e^{-s\,(n_j+C_0)\,d_{\mathcal U}} \,=\,
e^{-s\,C_0\,d_{\mathcal U}} \,\sum_{j\,\in\, J}\,
e^{-s \,d_{\mathcal U}\,n_j}.$$
Since $a < s\, d_{\mathcal U}$, we have
$$
M_Z(s,C\varepsilon,r_0) \,\leq\,
e^{-s\,C_0\,d_{\mathcal U}}\,
\sum_{j\,\in\, J}\,e^{-a\, n_j}\,
<\,\eta.$$
As $\eta>0$ is arbitrary, we conclude that
$ 
M_Z(s,C\varepsilon)=0,
$ 
and so
$$
h_{\mathrm{CP}}(Z,\mathcal F,C\varepsilon)\leq s.
$$
Letting $\varepsilon$ tend to zero one concludes that
$ 
h_{\mathrm{CP}}(Z,\mathcal F)\leq s.
$ 
Since $s$ is arbitrary, this yields
$$
h_{\mathrm{CP}}(Z,\mathcal F)
\,\leq\,
\sup_{\mathcal U}
\left\{
\frac{
h_{Z_{\mathcal U}}(G_{\mathcal U},G_{\mathcal U,1})
}
{\operatorname{diam}(\mathcal U)}
\right\},
$$
which completes the proof of equation ~\eqref{eq:equalitythmA}.

\smallskip

Assume now, in addition, that $Z$ is Borel. Then
$Z_{\mathcal U}=Z\cap X_{\mathcal U}$ is a Borel subset of
$X_{\mathcal U}$. Hence, by the variational principle for finitely
generated pseudogroups \cite[Theorem~B]{BCV},
$$
h_{Z_{\mathcal U}}(G_{\mathcal U},G_{\mathcal U,1})
\,=\,
\sup_{\substack{
\nu\,\in\,\mathcal P(X_{\mathcal U})\\
\nu(Z_{\mathcal U})\,=\,1
}}
\,\underline h_\nu(G_{\mathcal U},G_{\mathcal U,1}).
$$
Inserting this identity into the equation~\eqref{eq:equalitythmA} gives
$$
h_{\mathrm{CP}}(Z,\mathcal F)
\,=\, 
\sup_{\mathcal U}
\sup_{\substack{
\nu\,\in\,\mathcal P(X_{\mathcal U})\\
\nu(Z_{\mathcal U})\,=\,1
}}
\,\left\{
\frac{
\underline h_\nu(G_{\mathcal U},G_{\mathcal U,1})
}
{\operatorname{diam}(\mathcal U)}
\right\}.
$$

Finally, taking $Z=M$, we have $Z_{\mathcal U}=X_{\mathcal U}$ for every
finite nice foliated atlas $\mathcal U$. The previous formulas yield
$$
h_{\mathrm{CP}}(M,\mathcal F)
\,=\,
\sup_{\mathcal U}
\left\{
\frac{
h_{X_{\mathcal U}}(G_{\mathcal U},G_{\mathcal U,1})
}
{\operatorname{diam}(\mathcal U)}
\right\}
\,=\,
\sup_{\mathcal U}
\sup_{\nu\,\in\,\mathcal P(X_{\mathcal U})}
\,\left\{
\frac{
\underline h_\nu(G_{\mathcal U},G_{\mathcal U,1})
}
{\operatorname{diam}(\mathcal U)}
\right\}.
$$
\hfill $\square$

\section{ Harmonic measures}\label{sec:harmonicm}

Throughout this section, we assume that $\mathcal F$ is sufficiently
regular so that the leafwise Laplacian $\Delta_{\mathcal F}$ and the
leafwise heat diffusion are well defined (e.g. $\mathcal F$ of class $C^2$).
\smallskip

Recall that a Borel probability measure $m\in \mathcal P(M)$ on a foliated Riemannian manifold is called \emph{harmonic} if it is invariant under the leaf heat flow. More precisely, on a foliated manifold $(M, \mathcal{F})$, for every $p \in M$ the leaf $L_{p}$ inherits a Riemannian metric from the ambient manifold; the leaf Laplacian, $\Delta _{\mathcal{F}}$, acts on smooth functions by taking the intrinsic Laplacian restricted to the leaf passing through each point, that is,
$$\Delta _{\mathcal{F}}\,f(p)\,=\,\Delta _{L_{p}}(f|{}_{L_{p}})(p).$$
A Borel probability measure $m$ is harmonic if, for all compactly supported $C^2$ differentiable functions $f\colon M \to \mathbb{R}$, one has
$$\int _{M}\Delta _{\mathcal{F}}\,f\,d\mu \,=\,0.$$

\subsection{Tame harmonic measures} 

In~\cite{Garnett}, Garnett proved that harmonic measures on foliated manifolds always exist and are fixed points of diffusion operators. By Rohlin's disintegration theorem \cite{Rohlin}, any Borel probability measure on the manifold can be disintegrated on foliated boxes, giving rise to leaf probability measures and a probability measure on transversals. Harmonic measures turn out to have the following special disintegration on foliated boxes.

\begin{proposition}\label{prop:disintegrationharm} \cite[Proposition~5.2]{Candel}
A  measure $m \in \mathcal P(M)$ is harmonic if and only if it admits a decomposition on any given foliated box $D\times X$ satisfying the following condition: if $\varphi \colon M \to \mathbb R$ is a continuous function with support in $D\times X$, then 
\begin{equation}\label{eq:disintegrationHarmonic}
\int_M \varphi \; dm \,= \,\int_X \;  \Big(  \int_{D\times\{x\}}  \varphi(p,x) \,h_x(p) \,d\Leb_x(p) \Big)\; d\nu(x)
\end{equation}
where $\nu$ is a Borel probability measure on the transversal $Z$, $h_x(\cdot)$ is $\Delta^*$-harmonic on $D\times\{x\}$ for $\nu$-almost every $x\in X$, $\Delta^*$ is the adjoint operator of $\Delta$ and $\Leb_x$ denotes the volume density on the plaque $D\times \{x\}$.
\end{proposition}

In particular, harmonic measures induce transverse measures in foliated
boxes. Conversely, transverse invariant measures give rise to harmonic
measures with constant leafwise densities. However, an arbitrary transverse probability measure may not be induced by a harmonic measure, as stated in Theorem~\ref{thm-non-liftable}.

\smallskip

There is a special class of harmonic measures for which the leaf harmonic densities $h_z$ are constant along the leaves. If this is the case, it is not hard to check that the harmonic measure induces a probability measure on transversals that is invariant by the holonomy pseudogroup. This class of harmonic measures suggests the following notion.

\begin{definition}\label{def:tame-harmonic}
\emph{We say that a harmonic measure $m$ is \emph{tame} if for every nice foliated atlas $\mathcal U$ there exist positive constants $\alpha_{_\mathcal U}, \,\beta_{_\mathcal U} > 0$ such that the harmonic densities $h_x$ defined by ~\eqref{eq:disintegrationHarmonic} satisfy 
$$\alpha_{_\mathcal U}\,\leq \, h_x \,\le\, \beta_{_\mathcal U}\quad \quad \,\text{for $\nu$-almost every $x \in X$}.$$}
 \end{definition}

\section{Local metric entropy}\label{sec:measure-geometric-entropy-F}

The next lemma states that the local metric entropy 
\begin{equation*}
    h_\mu(M,\cF,\cU)=\int_{X_\cU} h_\mu(M,\mathcal F,x)\, d\nu_{X_\cU}(x)
\end{equation*} 
of a probability measure $\mu$ on $M$ with respect to an atlas $\cU$ is independent of the transverse section.  

\begin{lemma}\label{le:changetransv}
Let $\mu$ be a probability measure on $M$ and  $\mathcal U$ be a finite foliated atlas. Consider two complete transversals $X_{\mathcal U}$ and $X'_{\mathcal U}$ associated with $\mathcal U$, and let $\nu$ and $\nu'$ be the transverse probability measures induced by the disintegration of $\mu$ on $X_{\mathcal U}$ and
$X'_{\mathcal U}$, respectively. Then
$$
\int_{X_{\mathcal U}} h_\mu(M,\mathcal F,x) \,d\nu(x)
\,=\,
\int_{X'_{\mathcal U}} h_\mu(M,\mathcal F,x)\,d\nu'(x).
$$
\end{lemma}

\begin{proof}
As $X_{\mathcal U}$ decomposes into finitely many holonomy domains, we may assume that there is a well-defined holonomy map $H \colon X_{\mathcal U}\to X'_{\mathcal U}.$ 
By Rohlin's disintegration theorem,
for $\mu$-almost every $x$ there exist probability measures $\gamma_x, \gamma_x'$ supported on the plaque $P_x$ such that 
$$
\mu_{\mid_U}
\,=\,
\int_{X_{\mathcal U}}\gamma_x\,d\nu(x)
\,=\,
\int_{X'_{\mathcal U}}\gamma'_z\,d\nu'(z).
$$
In particular, if $A\subset X'_{\mathcal U}$ is a Borel set and
denote by $\operatorname{Sat}(A)= \bigcup_{x\,\in \,A} P_x$
the measurable foliated sets of 
obtained by its saturation by plaques inside the
corresponding foliated boxes, then
$ 
\operatorname{Sat}(A)=\operatorname{Sat}(H^{-1}(A)).
$ 
Using the two disintegrations of $\mu$,
$
\nu'(A)
=
\mu(\operatorname{Sat}(A))
=
\mu(\operatorname{Sat}(H^{-1}(A)))
=
\nu(H^{-1}(A))
$ 
and, consequently,
\begin{equation}\label{eq:reducH}
\int_{X_\cU'} h_\mu(M,\mathcal F,x)\, d\nu'(x)
=
\int_{X_\cU} h_\mu(M,\mathcal F,H(x))\, d\nu(x).    
\end{equation}
for every Borel subset $A\subset X_\cU'$.

\smallskip
\begin{claim}\label{cl:3}
\emph{$h_\mu(M,\mathcal F,H(x)) = h_\mu(M,\mathcal F,x)$ at $\nu$-almost every $x \in X_U$.}  
\end{claim}

\begin{proof}[Proof of Claim~\ref{cl:3}.]
Fix a point $x$ in a holonomy domain. Since $x$ and $Hx$ lie on the same plaque and the atlas is finite, there exists a constant $C>0$, independent
of $x$ ($C=\diam \mathcal U >0$) such that $x$ and $H(x)$ can be joined inside the leaf by a curve of
leaf length at most $C$. Moreover, by the uniform continuity of the finitely many local holonomy maps, for every small $\vep>0$ there exists $0 < \delta = \delta_\vep < \vep$ such that, for every $r>0$,
\begin{equation}\label{eq:inclusoesballs}
B_{r+C}(x,\delta)
\subset
B_r(H(x),\vep)
\quad\text{and}\quad
B_{r+C}(H(x),\delta)
\subset
B_r(x,\varepsilon).
\end{equation}
The second inclusion in ~\eqref{eq:inclusoesballs} ensures that 
\begin{align*}
\liminf_{r\,\to\,+\infty} \,-\frac{1}{r} \,\log 
\mu(B_{r}(H(x),\delta))
& \,\leq\,\, \liminf_{r\,\to\,+\infty}
-\frac{1}{r}\,
\log \mu(B_{r-C}(x,\vep)) \\
& \,=\,\, \liminf_{r\,\to\,+\infty}
-\frac{r-C}{r}
\frac{1}{r-C}\,
\log \mu(B_{r-C}(x,\vep)) \\
&\, =\,\, \liminf_{r\,\to\,+\infty}
-\frac{1}{r-C}\,
\log \mu(B_{r-C}(x,\vep)) \\
& \,=\,\,
\liminf_{R\,\to\,+\infty}
-\frac1R\, \log \mu(B_R(x,\vep)).\\
\end{align*}
Taking $\vep\to0^+$ (so $\delta\to0^+$),  we obtain
$
h_\mu(M,\mathcal F,H(x))
\leq
h_\mu(M,\mathcal F,x).
$
Applying the same argument to the inverse holonomy map $H^{-1}$ we obtain the converse inequality. Thus,
$$
h_\mu(M,\mathcal F,H(x))\, = \,h_\mu(M,\mathcal F,x)
$$ 
for $\nu$-almost every $x\in X_{\mathcal U}$.
\end{proof}

\bigskip

Since $\nu'=H_*\nu$, from Claim~\ref{cl:2} we obtain
\begin{equation}\label{eq:H}
\int_{X'_{\mathcal U}} h_\mu(M,\mathcal F,z)\,d\nu'(z)
=
\int_{X_{\mathcal U}} h_\mu(M,\mathcal F,H(x))\,d\nu(x)
=
\int_{X_{\mathcal U}} h_\mu(M,\mathcal F,x)\,d\nu(x).
\end{equation}
Lemma~\ref{le:changetransv} now follows directly from \eqref{eq:H} together with \eqref{eq:reducH}.
\end{proof}
\smallskip

\begin{remark}\label{ondefharmonicmu}
\emph{The previous argument also shows that the function
$x\mapsto h_\mu(M,\mathcal F,x)$ is invariant under local holonomy maps. Thus, the quantity
$ 
h_\mu(M,\mathcal F,\mathcal U)
=
\int_{X_{\mathcal U}} h_\mu(M,\mathcal F,x)\,d\nu_{\mathcal U}(x)
$ 
depends only on the atlas $\mathcal U$ and not on the particular complete  transversal chosen to represent it.}
\end{remark}


\section{Proof of Theorem~\ref{thm-non-liftable}}\label{se:thmE}

Let us now discuss the relation between measures on transverse sections and harmonic measures.

\subsection{Induced measures on transverse sections}
It is known that transversal invariant probability measures embed naturally in the space of harmonic measures, since the constant functions are leaf harmonic (cf. Corollary~4.7.10 in \cite{Walczak}). It is natural to ask which Borel probability measures can be induced by harmonic measures on sections transversal to the foliation.

\subsection{Liftability of harmonic measures}

Given a finite nice foliated atlas, every harmonic measure induces a probability measure on the associated complete transversal. However, the converse is false in general: an arbitrary transverse probability measure may not arise from a harmonic probability measure on the foliated manifold. Let us see an example.
\smallskip

Consider the flat torus $M=\mathbb T^2=\mathbb R^2/\mathbb Z^2$ and let $\mathcal F$ be the linear foliation generated by the constant unit vector field
$$
X=
\frac{1}{\sqrt{1+\alpha^2}}
\left(
\frac{\partial}{\partial x}
+
\alpha\frac{\partial}{\partial y}
\right),
$$
where $\alpha\notin\mathbb Q$. The leaves of $\mathcal F$ are dense, and so $\mathcal F$ is minimal.
Moreover, $\Sigma=\{x=0\}\cong\mathbb S^1$ is a global section transversal to the foliation $\mathcal F$. The holonomy group on $\Sigma$ is generated by the first return Poincar\'e map, which is the irrational rotation
$$R_\alpha(y)=y+\alpha \mod 1.$$

\begin{claim}\label{cl:2}
\emph{The only harmonic probability measure for $\mathcal F$ is the Haar measure on $\mathbb T^2$.}
\end{claim}

\begin{proof}[Proof of Claim~\ref{cl:2}.]
We shall use the uniqueness of Fourier--Stieltjes coefficients on compact abelian groups. More precisely, if $\mu$ is a finite Borel measure on $\mathbb T^2$ and
$$
\widehat\mu(k)
=
\int_{\mathbb T^2}
e^{2\pi i\,(k_1x\,+\,k_2y)}\,d\mu(x,y),
\qquad
k=(k_1,k_2)\in\mathbb Z^2
$$
denotes its Fourier coefficients, then the set
$\{\widehat\mu(k)\colon k\in\mathbb Z^2\}$ uniquely determines $\mu$ (cf. \cite[Chapter~1]{RudinFourierAnalysisGroups}).
In particular, if $\mu$ is a probability measure and
$\widehat\mu(k)=0$ for every $k\neq 0$, then $\widehat\mu(0)=1$ and $\mu$ has the same Fourier coefficients as the Haar probability measure on $\mathbb T^2$, and so $\mu$ is the Haar measure. 
\smallskip

Let $\mu$ be a harmonic probability measure for $\mathcal F$. We will show that $\widehat\mu(k)=0$ for every $k\neq 0$.
Since the leaves of $\mathcal F$ are one-dimensional and the vector field $X$ has unit norm along the leaves, the leafwise Laplacian is the second derivative in the leaf direction. Thus, for every smooth function $f$,
$$
\Delta_{\mathcal F}f=X(Xf)=X^2f
$$
(here $X^2$ denotes the composition of the derivation $X$ with itself). Consider the smooth functions
$\varphi_k(x,y)=e^{2\pi i\,(k_1x\,+\,k_2y)}$ on the torus, parameterized by $k=(k_1,k_2)\in\mathbb Z^2$, 
and their second directional derivatives 
$$
X^2\varphi_k
=
-\frac{4\pi^2}{1+\alpha^2}\,
(k_1+\alpha k_2)^2 \, \varphi_k.
$$
Since $\mu$ is harmonic, 
$$
0
=
\int_{\mathbb T^2}X^2\varphi_k\,d\mu
\,=\,
-\frac{4\pi^2}{1+\alpha^2}\,
(k_1+\alpha k_2)^2
\int_{\mathbb T^2}\varphi_k\,d\mu.
$$
Since $\alpha$ is irrational, $k_1+\alpha k_2\neq0$ for every
$k\neq0$, which guarantees that
$$
\int_{\mathbb T^2}\varphi_k\,d\mu=0
\qquad
\text{for every } k\neq0.
$$
Thus, all non-zero Fourier coefficients of $\mu$ vanish, and so $\mu$ is the Haar measure on $\mathbb T^2$. 
\end{proof}

The Haar measure on $\mathbb T^2$ induces the Lebesgue measure on the transversal $\Sigma \cong \mathbb S^1$. Therefore, for any probability measure $\nu \neq \mu$ on $\Sigma$ (e.g. $\nu=\delta_0\in\mathcal P(\Sigma)$) there is no harmonic probability measure on $(M,\mathcal F)$ inducing $\nu$ on $\Sigma$. 

\section{Proof of Theorem~\ref{thm:harmonic-measure-entropy}}
\label{sec:proof-harmonic-measure-entropy}

Let $\mathcal F$ be a $C^2$ non-singular foliation of a compact Riemannian manifold $M$ with a global transversal, and let
$\mu\in\mathcal P_H^{\mathrm{tame}}(M,\mathcal F)$. Given a finite nice foliated atlas $\mathcal U$, denote by $\nu_{\mathcal U}$ the probability
measure induced by $\mu$ on the global transversal $X_{\mathcal U}$. Since
$$
h_\mu(M,\mathcal F)
=
\sup_{\mathcal U} h_\mu(M,\mathcal F,\mathcal U),
$$
in order to prove that 
$$h_\mu(M,\mathcal F) \,=\,\sup_{\mathcal U}\,\Big\{\frac{\underline{h}_{\nu_{\mathcal U}}(G_{\mathcal U},G_{\mathcal U,1})}{\operatorname{diam}(\mathcal U)}\Big\} 
$$
it is enough to show that, for every finite nice foliated atlas $\mathcal U$ of $M$, one has 
$$h_\mu(M,\mathcal F,\mathcal U)
\,=\,\,\frac{\underline{h}_{\nu_{\mathcal U}}(G_{\mathcal U},G_{\mathcal U,1})}{\operatorname{diam}(\mathcal U)}.$$


\begin{proposition}
\label{prop:2estimates}
Under the assumptions of Theorem~\ref{thm:harmonic-measure-entropy}, one has:
\begin{itemize}
\item[(a)] For every finite nice foliated atlas $\mathcal U$ of $M$, 
\begin{equation}\label{eq:geq}
h_\mu(M,\mathcal F,\mathcal U)
\,\geq\,\,\frac{\underline{h}_{\nu_{\mathcal U}}(G_{\mathcal U},G_{\mathcal U,1})}{\operatorname{diam}(\mathcal U)}.
\end{equation}

\item[(b)] For every $\zeta>0$ and any finite nice foliated atlas $\mathcal U$ of $M$, there exists a finite nice foliated atlas $\mathcal U_\zeta$ that refines $\mathcal U$ so that
$$h_\mu(M,\mathcal F,\mathcal U)
\,\leq\,\,
\frac{\underline{h}_{\nu_{\mathcal U_\zeta}}(G_{\mathcal U_\zeta},G_{\mathcal U_\zeta,1})}{\operatorname{diam}(\mathcal U_\zeta)} + \zeta.$$
\end{itemize}
\end{proposition}
 
\begin{proof}    
$(a)$ Fix a nice foliated atlas $\mathcal U$ and let $\nu_{\cU}$ denote the transverse probability measure induced by $\mu$ on a global transversal $X_\cU$. Recall that 
$$
h_\mu(M,\mathcal F,\mathcal U)
=
\int_{X_\cU} h_\mu(M,\mathcal F,x)\, \,d\nu_\cU
$$
where 
\begin{equation*}
h_\mu(M,\mathcal F,x) \,=\,\lim_{\vep\,\to\,0^+}\,
\liminf_{r\,\to\,+\infty}\, 
- \frac1r \,\log \mu(B_r(x,\vep)).
\end{equation*}

\begin{lemma}
\label{lem:foliated-to-holonomy-balls}
Given a finite nice foliated atlas $\mathcal U$, there exist constants $C,C_0>0$ and $\varepsilon_{\mathcal U}>0$ such that, for every
$x\in X_{\mathcal U}$, every $r>0$ and every
$0<\varepsilon<\varepsilon_{\mathcal U}$, if
$ 
n=\lfloor \frac{r}{\operatorname{diam}(\mathcal U)}\rfloor
$ 
then
$$
B_r(x,\varepsilon)\,
\subset
\bigcup_{y\,\in\, B_n^{G_{\mathcal U}}(x,C\varepsilon)}
B_{\mathcal F}^{\mathrm{loc}}(y,C_0\,\varepsilon),
$$
where $B_{\mathcal F}^{\mathrm{loc}}(y,C_0\,\varepsilon)$ stands for the 
leafwise ball inside the local plaque containing $y$, centered at $y$ and with radius $C_0\,\varepsilon$.
\end{lemma}

\begin{proof}
Since $\mathcal U$ is finite and nice, there are constants $C_0>0$ and $\varepsilon_{\mathcal U}>0$ such that, whenever
$d_M(z,X_{\mathcal U})<\varepsilon_{\mathcal U}$, the local plaque through $z$ meets $X_{\mathcal U}$ in a unique point $y$ and
$$ 
d_{\mathcal F}^{\mathrm{loc}}(y,z)\,\leq\, C_0\,d_M(y,z).
$$
Taking a smaller $\varepsilon_{\mathcal U}$ if necessary, we may assume that this estimate applies to every $z$ considered below.

\smallskip
Fix $x\in X_{\mathcal U}$ and let $z\in B_r(x,\varepsilon)$, where 
$0<\varepsilon<\varepsilon_{\mathcal U}$. Since
$D_r(x,z)<\varepsilon$, in particular $d_M(x,z)<\varepsilon$. Let $y\in X_{\mathcal U}$ be the intersection of the local plaque through $z$ with the local transversal. By the comparison between leafwise dynamical balls and holonomy dynamical balls in \cite[Lemme~3.7]{GLW} (see also
\cite[Theorem~3.4.1]{Walczak}), it follows that 
$ 
z\in B_{\mathcal F}^{\mathrm{loc}}(y,C_0\,\varepsilon).
$ 

\smallskip
It remains to show that $y\in B_n^{G_{\mathcal U}}(x,C\varepsilon)$,
where
$ 
n=\lfloor\frac{r}{\operatorname{diam}(\mathcal U)}\rfloor.
$ 
Let $g\in G_{\mathcal U,n}$ be an admissible holonomy word defined at
$x$. This word is represented by a chain of at most $n$ plaques, and hence
by a leafwise curve of length at most
$n\operatorname{diam}(\mathcal U)$, where $n\operatorname{diam}(\mathcal U)\leq r.$ 
Since $D_r(x,z)<\varepsilon$, the corresponding plaques through $z$ stay
uniformly close to those through $x$ up to this length. By the uniform
continuity of the finitely many local holonomy maps, there exists $C>0$, depending only on $\mathcal U$, such that
$ 
d(g(x),g(y))<C\varepsilon
$ 
for every admissible $g\in G_{\mathcal U,n}$. Consequently, 
$ 
y\in B_n^{G_{\mathcal U}}(x,C\varepsilon).
$
\end{proof}

\smallskip
Let us prove item (a) of Proposition~\ref{prop:2estimates}.
Fix a finite nice foliated atlas $\mathcal U$ and let $\nu_{\mathcal U}$ denote the transverse probability measure induced by $\mu$ on $X_{\mathcal U}$. Given $x\in X_{\mathcal U}$, $r>0$ and
$0<\varepsilon<\varepsilon_{\mathcal U}$, set
$ 
n=\lfloor \frac{r}{\operatorname{diam}(\mathcal U)}\rfloor
$ 
as before. By Lemma~~\ref{lem:foliated-to-holonomy-balls},
$$
B_r(x,\varepsilon)
\,\subset
\bigcup_{y\,\in\, B_n^{G_{\mathcal U}}(x,C\varepsilon)}\,
B_{\mathcal F}^{\mathrm{loc}}(y,C_0\,\varepsilon).
$$
Using the disintegration formula 
~\eqref{eq:disintegrationHarmonic} and the tameness of $\mu$, we obtain
$$
\begin{aligned}
\mu(B_r(x,\varepsilon))
&\,\leq\,
\int_{B_n^{G_{\mathcal U}}(x,\,C\varepsilon)}
\left(
\int_{B_{\mathcal F}^{\mathrm{loc}}(y,\,C_0\,\varepsilon)}
h_y(z)\,dz
\right)
d\nu_{\mathcal U}(y) \\
&\,\leq\,
\beta_{\mathcal U}\,
\sup_{y\,\in\, X_{\mathcal U}}\,
\operatorname{vol}
\bigl(
B_{\mathcal F}^{\mathrm{loc}}(y,C_0\,\varepsilon)
\bigr)
\,
\nu_{\mathcal U}
\bigl(
B_n^{G_{\mathcal U}}(x,C\varepsilon)
\bigr).
\end{aligned}
$$
Since the atlas is finite, there exists $K>0$ such that, for every sufficiently small $\varepsilon>0$,
$$ 
\sup_{y\,\in\, X_{\mathcal U}}
\,\operatorname{vol}
\bigl(
B_{\mathcal F}^{\mathrm{loc}}(y,C_0\,\varepsilon)
\bigr)
\,\leq\,
K(C_0\,\varepsilon)^{\dim\mathcal F}
$$ 
and so
$$
\mu(B_r(x,\varepsilon))
\,\leq\,
\beta_{\mathcal U}\, K(C_0\,\varepsilon)^{\dim\mathcal F}
\,
\nu_{\mathcal U}
\bigl(
B_n^{G_{\mathcal U}}(x,C\varepsilon)
\bigr).
$$ 
Combined with the fact that 
$ 
\frac{n}{r}\to
\frac1{\operatorname{diam}(\mathcal U)}
$
as
$r$ goes to $+\infty$,  
the latter inequality implies that 
$$
\liminf_{r\,\to\,+\infty}\,
-\frac1r\log\mu(B_r(x,\varepsilon))
\,\geq\,
\frac1{\operatorname{diam}(\mathcal U)}\,
\liminf_{n\,\to\,+\infty}\,
-\frac1n
\log
\nu_{\mathcal U}
\bigl(
B_n^{G_{\mathcal U}}(x,C\varepsilon)
\bigr).
$$
Letting $\varepsilon\to0^+$, we get
$$ 
h_\mu(M,\mathcal F,x)
\,\geq\,
\frac1{\operatorname{diam}(\mathcal U)}
h_{\nu_{\mathcal U}}
(G_{\mathcal U},G_{\mathcal U,1},x)
$$ 
for $\nu_{\mathcal U}$-almost every $x\in X_{\mathcal U}$. Finally, integrating with respect to $\nu_{\mathcal U}$ we obtain the  inequality stated in item (a).

\bigskip

We proceed by proving item $(b)$ of Proposition~\ref{prop:2estimates}. We will apply the next lemma to bound from above the local metric entropy of $\mu$.

\begin{lemma}\label{le:ball-comparison2}
There is a constant $C>0$ such that, given a finite nice foliated atlas $\mathcal U$ and $\zeta>0$, there exists a finite nice foliated atlas $\cU_\zeta$ refining $\cU$ and satisfying
$$\bigcup \Big\{ B_{\cF}(y,\, \vep/4) \colon 
y\in B^G_{\big\lfloor r/[(1-\eta)\diam{\cU_\zeta}] \big\rfloor}(x,C\vep/2)\Big\}
\,\subseteq \,B_r(x,\vep)$$
for every sufficiently small $\varepsilon>0$, 
$0<\eta<1$, $r>0$ and $x\in X_{\mathcal U_\zeta}$.
\end{lemma}

\begin{proof}
The argument follows closely the proofs of \cite[Lemme 3.11]{GLW} and \cite[Theorem 3.4.1]{Walczak}. Firstly, \cite[Theorem 3.4.1]{Walczak} implies that, given $\zeta > 0$ and a finite nice foliated atlas $\cU$, there exists a finite nice foliated atlas $\cU_\zeta$, obtained by refining $\cU$, such that all plaques of $\cU_\zeta$ have the same diameter and 
$$h_\mu(M,\cF,\cU) \, \leq \, h_\mu(M,\cF,\cU_\zeta) + \zeta.$$
Moreover, given $\vep>0$ and $0<\eta \ll \zeta$ sufficiently small, if $r=(1-\eta)n\zeta$ for some $n \in \N$ and $x,y \in X_{\cU_\zeta}$ satisfy $d_M(x,y)<\vep$ and $D_r(x,y)\geq\vep$, then there exists $\underline{g}\in G^{\cU_\zeta}_n$ such that $d(\underline{g}(x),\,\underline{g}(y))>C\vep$, where $C>0$ is a uniform constant depending only on the foliation $\cF$ (cf. \cite[page 81]{Walczak}). That is, 
$$d(g(x),\,g(y))\,\le\, C\,\vep \quad  \forall \,g\in G^{\cU_\zeta}_n \quad  \Rightarrow \quad   D_r(x,y)<\vep.$$
Consequently, 
$$
 B^{G^{\cU_\zeta}}_{\big\lfloor r/[(1-\eta)\diam{\cU_\zeta}] \big\rfloor}\big(x,C\vep/2\big) \,\subseteq \,B_r(x,\vep/2) \cap X_{\cU_\zeta}.$$
Since the leaf diameter of 
$B_{\cF}(y,\vep/4)$ is $\vep/2$ for each $y\in B^G_{\lfloor r/[(1-\eta)\diam{\cU_\zeta}] \rfloor}\big(x,C\vep/2\big)$, the result follows from the triangular inequality.
\end{proof}

For every $x\in X_{\mathcal U_\zeta}$, set
$ 
n_r=
\lfloor
\frac{r}{(1-\eta)\operatorname{diam}(\mathcal U_\zeta)}
\rfloor.
$ 
Using Lemma~~\ref{le:ball-comparison2}, the disintegration formula and the tameness of $\mu$, we conclude that
$$
\mu(B_r(x,\varepsilon))
\,\geq\,
\mu\Big(
\bigcup_{y\,\in\, B_{n_r}^{G^{\mathcal U_\zeta}}(x,C\varepsilon/2)} \,\,
B_{\mathcal F}^{\mathrm{loc}}(y,\varepsilon/4)
\Big)  
\,\geq\,
\alpha_{\mathcal U_\zeta}
K^{-1}
(\varepsilon/4)^{\dim\mathcal F}
\,
\nu_{\mathcal U_\zeta}
\left(
B_{n_r}^{G^{\mathcal U_\zeta}}(x,C\varepsilon/2)
\right).$$
Therefore, as
$ 
\frac{n_r}{r}
\to
\frac1{(1-\eta)\operatorname{diam}(\mathcal U_\zeta)},
$ 
taking $\liminf$ as $r\to+\infty$
we get 
$$
\begin{aligned}
-\frac1r\log\mu(B_r(x,\varepsilon))
&\leq
-\frac1r
\log\left(
\alpha_{\mathcal U_\zeta}
K^{-1}
(\varepsilon/4)^{\dim\mathcal F}
\right) \\
&\quad
+
\frac{n_r}{r}
\left(
-\frac1{n_r}
\log
\nu_{\mathcal U_\zeta}
\left(
B_{n_r}^{G^{\mathcal U_\zeta}}(x,C\varepsilon/2)
\right)
\right).
\end{aligned}
$$
Letting $\varepsilon$ go to $0^+$, we obtain 
$$ 
h_\mu(M,\mathcal F,x)
\,\leq\,
\frac1{(1-\eta)\operatorname{diam}(\mathcal U_\zeta)}\,
h_{\nu_{\mathcal U_\zeta}}
(G^{\mathcal U_\zeta},G^{\mathcal U_\zeta}_1,x).$$ 
Integrating with respect to $\nu_{\mathcal U_\zeta}$, we get
$$
h_\mu(M,\mathcal F,\mathcal U_\zeta)
\,\leq\,
\frac1{(1-\eta)\operatorname{diam}(\mathcal U_\zeta)}\,
h_{\nu_{\mathcal U_\zeta}}
(G^{\mathcal U_\zeta},G^{\mathcal U_\zeta}_1)$$
so, since $\eta>0$ is arbitrary, it follows that
$$
h_\mu(M,\mathcal F,\mathcal U_\zeta)
\,\leq\,
\frac1{\operatorname{diam}(\mathcal U_\zeta)}\,
h_{\nu_{\mathcal U_\zeta}}
(G^{\mathcal U_\zeta},G^{\mathcal U_\zeta}_1).$$
This estimate, together with \eqref{eq:geq} applied to
$\mathcal U_\zeta$, yields
$$
h_\mu(M,\mathcal F,\mathcal U_\zeta)
\,=\,
\frac1{\operatorname{diam}(\mathcal U_\zeta)}
h_{\nu_{\mathcal U_\zeta}}
(G^{\mathcal U_\zeta},G^{\mathcal U_\zeta}_1).
$$
\end{proof}

Given $\zeta > 0$, Proposition~\ref{prop:2estimates} implies that
\begin{equation*}\label{vp:2}
h_\mu(M,\mathcal F)  \,=\, \sup_{\mathcal U}\, h_\mu(M,\mathcal F, \mathcal U) 
\,\leq\,\sup_{\mathcal U}\,\left(\frac{\underline{h}_{\nu_{\mathcal U_\zeta}}(G_{{\mathcal U}_\zeta},G_{{\mathcal U}_\zeta,1})}{\operatorname{diam}({\mathcal U}_\zeta)} + \zeta\right) \,\leq\,\sup_{\mathcal U}\,\Big\{\frac{\underline{h}_{\nu_{\mathcal U}}(G_{\mathcal U},G_{\mathcal U,1})}{\operatorname{diam}(\mathcal U)}\Big\} + \zeta.
\end{equation*}
Since $\zeta$ is arbitrarily small, it follows that
$$h_\mu(M,\mathcal F)  \,\leq\,\sup_{\mathcal U}\,
\Big\{\frac{\underline{h}_{\nu_{\mathcal U}}(G_{\mathcal U},G_{\mathcal U,1})}{\operatorname{diam}(\mathcal U)}\Big\}.$$
This finishes the proof of Theorem~\ref{thm:harmonic-measure-entropy}.


\section{Examples}\label{sec:examples}

The examples in this section emphasize three aspects: the geometric entropy of foliations generated by flows recovers the classical topological entropy for dynamical systems; holonomy invariant measures may fail to detect the full geometric entropy of the foliation; suspended foliations are a source of relevant examples where the relation between leaf geometry and transverse holonomy can be evinced.

\subsection{Suspended foliations}\label{sec:suspended}

In this subsection, we recall the notion of suspended foliation following \cite{Richardson,Walczak}.  
Let $F$ and $B$ be compact smooth connected manifolds. Denote by $\pi_1(B)$ the fundamental group of homotopy of $B$ and by $\tilde B$ the universal cover of $B$. Assume that $\phi\colon  \pi_1(B)\to \text{Diff}^r(F)$ ($r \ge 1$)  is a representation of $\pi_1(B)$, that is, a group homomorphism. The leaves of the \emph{suspended foliation} $\mathcal F$ in the manifold 
$$M = F \times_\phi B = (F \times \tilde B) / \pi_1(B)$$
are immersed $\tilde B$-submanifolds such that $F$ is a global transversal to the foliation, where the action of $\pi_1(B)$ on $F \times \tilde B$ is given by 
$$\gamma \cdot (x,y) = (\phi(\gamma)(x), \gamma(y)) \quad \forall \, \gamma \in \pi_1(B), \,\, \forall \, (x,y) \in F \times \tilde B$$
and $\gamma$ acts as the deck transformation on $\tilde B$. The foliation can be described by
$$\mathcal F = \Big\{\{z\} \times \tilde B \colon z\in F \Big\} \,/\, \pi_1(B)$$
whose leaves are the projections of the sets $\{x\}\times\widetilde B$ to the quotient manifold $M$. 
By construction, the holonomy pseudogroup is a group isomorphic to the subgroup $\phi(\pi_1(B))$ of $\text{Diff}^{\, r}(F)$, and the transverse dynamics of the foliation is encoded by the action of $\phi(\pi_1(B))$ on $F$.
 Equivalently, the leaf $\mathcal F_x$ is naturally identified with
$ 
\widetilde B/\operatorname{Stab}(x),
$ 
where
$
\operatorname{Stab}(x)
=
\{\gamma\in\pi_1(B):\varphi(\gamma)(x)=x\}.
$ 

%

\begin{lemma}\label{le:lift-suspension}
Let $M=F\times_\varphi B$ be a suspended foliation. If $\nu$ is a Borel holonomy invariant transverse probability measure, then $\nu$ is the lifting of a tame harmonic probability measure on $(M,\mathcal F)$ whose leafwise densities are constant.
\end{lemma}

\begin{proof}
Let $\nu$ be a holonomy invariant transverse probability measure. On the
covering space $F\times\widetilde B$, consider the product measure
$ 
\nu\times \operatorname{vol}_{\widetilde B}.
$ 
The holonomy invariance of $\nu$, together with the invariance of the
Riemannian volume on $\widetilde B$ under deck transformations, ensures
that this product measure descends to a well-defined finite measure on
the suspension $M=F\times_\varphi B.$ 
After normalization, we obtain a probability measure $\mu_\nu$ on $M$.
In foliated boxes, the conditional measures of $\mu_\nu$ along plaques
are given by the Riemannian volume on the leaves. Equivalently, the
corresponding leafwise densities are constant. Therefore $\mu_\nu$ is a
harmonic measure. Since the densities are constant, they are uniformly
bounded from above and from below on every finite nice foliated atlas;
hence $\mu_\nu$ is tame.
\end{proof}

\begin{remark}
\emph{The holonomy-invariance assumption in Lemma~9.1 is essential. Without
it, the product measure constructed on the covering space need not
descend to the suspension. Therefore the lemma produces tame harmonic
measures from holonomy invariant transverse measures, but it does not
imply that suspended foliations are tame foliated spaces in the sense of
Definition~\ref{def:tamefoliation}.}
\end{remark}

\begin{remark}\label{rmk:realiation}
\emph{Since every finitely presented group can be realized as the fundamental group of a compact, connected, smooth manifold of dimension 4 or higher, 
a consequence of this construction is the following: any finitely generated group of diffeomorphisms of a manifold can be realized as a holonomy (pseudo)group of some foliation. We refer the reader to 
\cite{Richardson} for many other examples of suspended foliations.}
\end{remark}

The class of suspended foliations is extremely useful as it reduces several questions about geometric entropy of foliations to questions about the entropy of finitely generated groups or pseudogroups of diffeomorphisms. If the action of $\rho(\pi_1(B))$ on $F$ has positive entropy, then the corresponding suspended foliation exhibits positive transverse complexity; conversely, when the transverse action has zero entropy, suspended foliations provide natural examples with vanishing transverse complexity. 
\smallskip

In what follows, we discuss an example where the foliation and the holonomy pseudogroup appear in the context of smooth flows which admit a global cross section. In this case, the leaves of the foliation are the orbits of the flow and the transverse dynamics is described by the corresponding Poincaré return map. Thus, the geometric entropy of the foliation is naturally related to the classical topological entropy of the underlying flow. 

\begin{example}
\label{ex:suspensionAnosov}
Let
$ 
f:\mathbb T^d\to\mathbb T^d
$ 
be an Anosov diffeomorphism on the torus $\mathbb T^d$, $d\geq2$, and let
$ R \colon \mathbb T^d\to\mathbb R_+$ 
be a $C^1$ roof function. Consider the suspension manifold
$$M \,=\, \{(x,s)\in\mathbb T^d\times\mathbb R_+ \colon \,  0\leq s\leq R(x)\}/\sim
$$ 
where 
$$(x, R(x)) \sim (f(x),0) \quad \quad \forall\, x\in\mathbb T^d.$$
The suspension flow $(X_t)_{t\,\in\,\mathbb R}$ is given in local coordinates by
$$X_t(x,s)\,= \, \Big(f^n(x),\,t+s-\sum_{j=0}^{n-1}\,R(f^j(x))\Big)$$ 
where $n=n(x,t,s)\geq0$ is uniquely determined by the inequalities
$$\sum_{j=0}^{n-1} \,R(f^j(x))\,\leq\,t+s \,
<\,\sum_{j=0}^{n}\,R(f^j(x)).$$ 

\smallskip
Let $\mathcal F$ be the one-dimensional foliation whose leaves are the orbits of the flow $(X_t)_{t\,\in\,\mathbb R}$. In this case, the geometric entropy of $\mathcal F$ measures the exponential growth rate of distinguishable orbit segments. More precisely, if the vector field generating the flow is parametrized with unit speed, then the leaf ball of radius $r$ centered at $x$ corresponds to the orbit segment
$$\{X_t(x) \colon  -r\leq t\leq r\}$$ 
and the geometric entropy of the foliation is known to be related to the topological entropy of the flow
(cf. \cite[Th\'eor\`eme~3.2]{GLW}). 
\smallskip

Borel probability measures invariant by the flow provide natural examples of harmonic measures for the orbit foliation. Indeed, assume that the vector field $X$ generating the flow has unit norm along the leaves. Then the leaf Laplacian is given by
$ 
\Delta_{\mathcal F}=X^2
$ 
and, if $\mu$ is preserved by the smooth flow $(X_t)_{t\,\in\,\mathbb R}$, then
$ 
\int Xg\,d\mu=0
$ 
for every sufficiently regular function $g$. Taking $g=Xf$, we get
$ 
\int \Delta_{\mathcal F}f\,d\mu
=
\int X^2f\,d\mu
=
0.
$ 
Thus, $\mu$ is harmonic. If, in addition, the conditional densities of $\mu$ in local flow
boxes are uniformly bounded from above and from below, then $\mu$ belongs to the tame class considered above.
\end{example}

\subsection{Holonomy invariant measures that carry no relevant entropy}

The next example illustrates the fact that, even if holonomy invariant measures exist, they may be insufficient to describe the entropy of a foliation.

\begin{example}\label{ex:holonomy-invariant-zero-entropy}
Consider the two-sphere $\mathbb S^2$ and a pair of diffeomorphisms 
$g_1,g_2\colon \mathbb S^2\to\mathbb S^2$ 
such that $g_1$ is an Axiom A $C^1$-diffeomorphism satisfying $h_{\mathrm{top}}(g_1)>0$, the map $g_2$ is a Morse--Smale $C^1$-diffeomorphism such that
$\operatorname{Fix}(g_2)=\operatorname{Per}(g_2)=\{p,q\}$, and the two points $p$ and $q$ are fixed by $g_1$. Let $G\subset \text{Diff}^1(\mathbb S^2)$ be the group generated by
$G_1=\{Id, g_1,g_2,g_1^{-1},g_2^{-1}\}.$ 
Since $g_1$ has positive topological entropy, then
$$h_{\mathrm{CP}}(G,G_1)
\,\geq\,
h_{\mathrm{CP}}(g_1) 
\,=\, h_{\mathrm{top}}(g_1)
\,>\,0.$$
On the other hand, the space of $G$-invariant probability measures is (cf. \cite{BCV}) 
$$
\mathcal P_G(\mathbb S^2)
=
\left\{
t\delta_p+(1-t)\delta_q \colon \, t\in[0,1]
\right\}.
$$
Consequently, all $G$-invariant probability measures are supported on fixed points and have zero metric entropy. In particular,
$$
\sup_{\mu\,\in\, \mathcal P_G(\mathbb S^2)}\,
\underline h_\mu(G,G_1)
=0
$$
whereas
\begin{equation}\label{eq:entropcpex}
h_{\mathrm{CP}}(G,G_1)>0.
\end{equation}
This example shows that holonomy invariant measures may not detect all entropy phenomena along the foliation.
\end{example}

\begin{example}\label{ex:zero-transverse-entropy}
This example is built upon the previous one. Let $B$ be a compact manifold and 
$ 
\phi\colon \pi_1(B)\to G
$ 
be a representation whose image is the group $G$  (recall Remark~\ref{rmk:realiation}). Consider the suspended foliation $\mathcal F$ on
$M = \mathbb S^2\times_\phi \widetilde B.$ 
By construction, the leaves $\mathcal F_p$ and $\mathcal F_q$ through $p$ and $q$, respectively, are compact. Moreover, 
the harmonic measures induced by holonomy invariant transverse probability measures (cf. Lemma~\ref{le:lift-suspension}) are supported on the union of these two compact leaves and carry zero transversal entropy.
However, by combining ~\eqref{eq:entropcpex} with Theorem~\ref{thm:entropy-holonomy} we conclude that 
$$\hgeo(M,\mathcal F) \,>\,0.$$
This shows that a positive transverse entropy may coexist with a very small space of holonomy invariant measures. 

\end{example}

\subsection*{Acknowledgments}
A. Bi\'s was partially supported by a grant from the University of $\L$odz with reference 11/IDUB/DOS/2021. M. Carvalho was partially supported by CMUP, member of LASI, which is financed by national funds through FCT–Fundação para a Ciência e a Tecnologia, I.P., under the project UID/00144/2025 (https://doi.org/10.54499/UID/00144/2025). P. Varandas was partially supported by CIDMA under the Portuguese Foundation for Science and Technology (FCT, https://ror.org/00snfqn58) Multi-Annual Financing Program for R\&D Units, grants UID/4106/2025 and UID/PRR/4106/2025.

\end{document}